\documentclass[11pt]{article}
\usepackage{fullpage}
\usepackage{amsmath,amsfonts,amssymb,amsthm}
\usepackage{graphicx}
\usepackage{epstopdf}
\usepackage{color}
\usepackage{epsf}
\usepackage{verbatim}
\usepackage{enumitem}
\usepackage{hyperref}
\usepackage{subfig}
\usepackage{tikz}
\usetikzlibrary{arrows,calc,positioning,through,intersections,backgrounds,matrix,fit}

\newtheorem{theorem}{Theorem}[section]
\newtheorem{lemma}[theorem]{Lemma}
\newtheorem{proposition}[theorem]{Proposition}
\newtheorem{corollary}[theorem]{Corollary}

\newtheorem{example}[theorem]{Example}

\newtheorem{procedure}[theorem]{Procedure}
\newtheorem{definition}[theorem]{Definition}
\newtheorem{conjecture}[theorem]{Conjecture}
\newtheorem{observation}[theorem]{Observation}

\newtheorem{claim}[theorem]{Claim}

\newcommand{\ep}{\varepsilon}

\newcommand{\floor}[1]{\left\lfloor#1\right\rfloor}
\newcommand{\ceiling}[1]{\left\lceil#1\right\rceil}

\providecommand{\N}{\mathbb{N}}
\providecommand{\ec}{\alpha}
\DeclareMathOperator{\E}{\mathbb{E}}
\DeclareMathOperator{\im}{im}
\providecommand{\msd}{\delta^0}
\providecommand{\pairs}{\mathcal{P}}

\providecommand{\cpairs}{\mathcal{R}}

\newcommand{\rdeg}[2]{\vec{e}(#1, #2)}

\newcommand{\outn}[1]{N^+(#1)}

\title{\vspace{-.5in}
Semi-degree threshold for anti-directed Hamiltonian cycles}
\author{
Louis DeBiasio\thanks{Department of Mathematics, Miami University, Oxford, OH 45056 USA. E-mail address: debiasld@miamioh.edu.} 
\quad and \quad 
Theodore Molla\thanks{School of Mathematical Sciences and Statistics, Arizona State University, Tempe, AZ 85287, USA. E-mail address: tmolla@asu.edu. Research of this author is supported in part by NSA grant H98230-12-1-0212.}
}
\date{\today}

\begin{document}
\maketitle

\begin{abstract}
In 1960 Ghouila-Houri extended Dirac's theorem to directed graphs by proving that if $D$ is a directed graph on $n$ vertices with minimum out-degree and in-degree at least $n/2$, then $D$ contains a directed Hamiltonian cycle.  For directed graphs one may ask for other orientations of a Hamiltonian cycle and in 1980 Grant initiated the problem of determining minimum degree conditions for a directed graph $D$ to contain an anti-directed Hamiltonian cycle (an orientation in which consecutive edges alternate direction).  We prove that for sufficiently large even $n$, if $D$ is a directed graph on $n$ vertices with minimum out-degree and in-degree at least $\frac{n}{2}+1$, then $D$ contains an anti-directed Hamiltonian cycle.  In fact, we prove the stronger result that $\frac{n}{2}$ is sufficient unless $D$ is one of two counterexamples.  This result is sharp.
\end{abstract}

\section{Introduction}

A directed graph $D$ is a pair $(V(D), E(D))$ where $E(D)\subseteq V(D)\times V(D)$.  In this paper we will only consider loopless directed graphs, i.e. directed graphs with no edges of the type $(v,v)$.  An \emph{anti-directed cycle (path)} is a directed graph in which the underlying graph forms a cycle (path) and no pair of consecutive edges forms a directed path.  Note that an anti-directed cycle must have an even number of vertices.  Let ADP, ADC stand for anti-directed path and anti-directed cycle respectively and let ADHP, ADHC stand for anti-directed Hamiltonian path and anti-directed Hamiltonian cycle respectively.  
% TNM added definition of proper ADP
Call an ADP $P = v_1 \dotsc v_d$ \emph{proper} if $d$ is even
and $(v_1, v_2) \in E(P)$ and hence, $(v_{d-1}, v_d) \in E(P)$.
% TNM
Given an (undirected) graph $G$, let $\delta(G)$ be the minimum degree of $G$.  If $D$ is a directed graph, then $\delta(D)$ will denote the minimum degree of the underlying multigraph, i.e. the minimum total degree of $D$.  For a directed graph $D$, let $\delta^+(D)$ and $\delta^-(D)$ be the minimum out-degree and minimum in-degree respectively.  Finally, let $\delta^0(D)=\min\{\delta^+(D), \delta^-(D)\}$ and call this quantity the minimum semi-degree of $G$.

In 1952, Dirac \cite{Dir} proved that if $G$ is a graph on $n\geq 3$ vertices with $\delta(G)\geq n/2$, then $G$ contains a Hamiltonian cycle.  In 1960, Ghouila-Houri  extended Dirac's theorem to directed graphs.
\begin{theorem}[Ghouila-Houri \cite{GH}]
Let $D$ be a directed graph on $n$ vertices.  If $\delta^0(D)\geq n/2$, then $D$ contains a directed Hamiltonian cycle.
\end{theorem}
\noindent
(His original statement actually says that $\delta(D)\geq n$ is sufficient if $D$ is strongly connected.)

In 1973, Thomassen proved that every tournament on $2n\geq 50$ vertices contains an ADHC  \cite{Thom}. Since the total degree of every vertex in a tournament on $2n$ vertices is $2n-1$, Grant wondered if all digraphs on $2n$ vertices with total degree $2n-1$ have an ADHC.  So in 1980, Grant made the weaker conjecture (replacing total degree by semi-degree) that if $D$ is a directed graph on $2n$ vertices with $\delta^0(D)\geq n$, then $D$ contains an ADHC \cite{G}.  However, in 1983, Cai \cite{C} gave a counterexample to Grant's conjecture (see Figure \ref{CaiFig1}).

\begin{example}[Cai 1983]\label{Cai:example}
For all $n$, there exists a directed graph $D$ on $2n$ vertices with $\delta^0(D)= n$ such that $D$ does not contain an ADHC.
\end{example}

%\begin{figure}[ht]
%\centering
%\subfloat[Cai's example where $\delta(D)=n$ and $D$ contains no ADHC]{
%\scalebox{.7}{\input{Cai.pdf_t}}
%\label{CaiFig}
%}~~~~~~~~~~~
%\subfloat[A general example where $\delta(D)=n-1$ and $D$ contains no ADHC]{
%\scalebox{.7}{\input{ExtremalMain.pdf_t}}
%\label{GenExtremal}
%}
%\label{ExFigs}
%\caption[]{The solid arrows indicate all possible edges in the designated direction.  The shaded sets with the curved arrows indicate all possible directed edges.}
%\end{figure}

We define for each even $n$, a family of digraphs with minimum semi-degree $n/2-1$ which have no anti-directed cycle on $n$ vertices.  From this family, we define two digraphs with minimum semi-degree $n/2$ which have no anti-directed cycle on $n$ vertices (see Figure \ref{ExFigs}).

\begin{definition}\label{CaiF}
Let $n\geq 2$ be even and let $0\leq k\leq \frac{n}{2}$.  Let $F_{n,k}$ be a digraph on $n$ vertices where $\{X_1, X_2, Y_1, Y_2\}$ is a partition of the vertex set with $|X_1|=|X_2|=\frac{n}{2}-k$ and $|Y_1|=|Y_2|=k$ and where $(u,v)$ is an edge if and only if $u\neq v$ and
\begin{enumerate}
\item $u\in Y_i$ and $v\in Y_i\cup X_{i}$ for $i\in [2]$ or

\item $u\in X_i$ and $v\in Y_{3-i}\cup X_{3-i}$ for $i\in [2]$.
\end{enumerate}

Let $F^1_{n}$ be the digraph obtained from $F_{n,1}$ by adding the edges $(y_1, y_2)$ and $(y_2, y_1)$, where $y_i$ is the unique vertex in $Y_i$.

Let $F^2_{n}$ be the digraph obtained from $F_{n,1}$ by adding the edges $(y_1, y_2)$, $(y_2, x)$, and $(x, y_1)$, where $y_i$ is the unique vertex in $Y_i$ and $x$ is an arbitrary vertex in $X_1$.
\end{definition}

\begin{figure}[ht]
\centering
\subfloat[$F_{2n,k}$]{
\scalebox{.7}{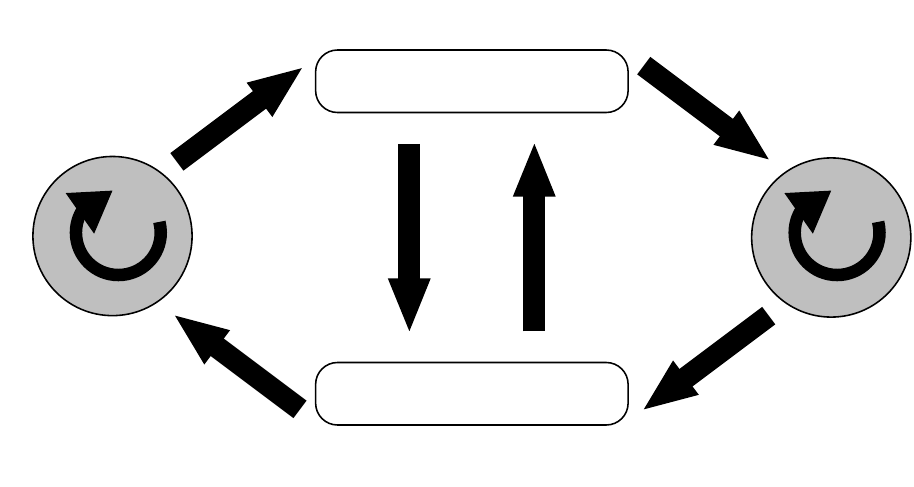}
\label{GenExtremal}
}~~~~~\\
\subfloat[$F^1_{2n}$]{
\scalebox{.7}{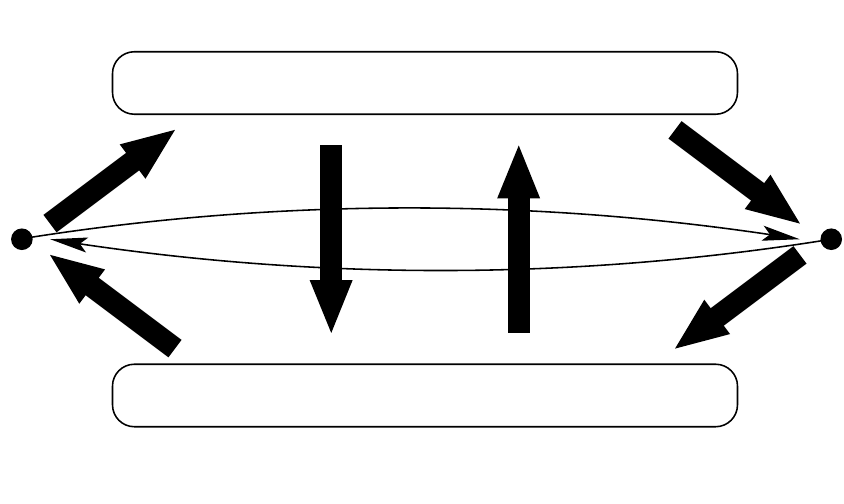}
\label{CaiFig1}
}~~~~~~~~~~~~~~~~~~~
\subfloat[$F^2_{2n}$]{
\scalebox{.7}{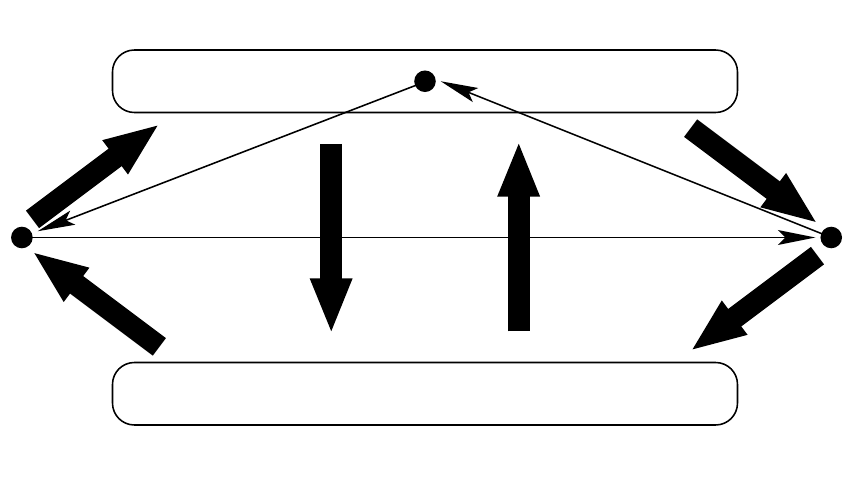}
\label{CaiFig2}
}
\caption[]{The solid arrows indicate all possible edges in the designated direction.  The shaded sets with the curved arrows indicate all possible directed edges.}\label{ExFigs}
\end{figure}

One can easily check that $F_{2n}^1$ and $F_{2n}^2$ have no ADHC and are edge maximal with respect to this property.  Cai's example ($F_{2n}^1$ above) and our modification of his example ($F_{2n}^2$ above) shows that the semi-degree threshold for an ADHC in a directed graph on $2n$ vertices is at least $n+1$.  There have been a sequence of partial results which have improved the threshold from the upper end.  In 1980, Grant proved that if $D$ is a directed graph on $2n$ vertices and $\delta^0(D)\geq \frac{4}{3}n+2\sqrt{n\log{n}}$, then $D$ has an ADHC \cite {G}.  In 1995, H\"aggkvist and Thomason proved the very general result that if $D$ is a directed graph on $n$ vertices then the semi-degree threshold for all orientations of a cycle on $n$ vertices is asymptotically $n/2$ (we conjecture an exact bound in Section \ref{conclusion}).

\begin{theorem}[H\"aggkvist, Thomason \cite{HT}]
For sufficiently large $n$, if $D$ is a directed graph on $n$ vertices and $\delta^0(D)\geq \frac{n}{2}+n^{5/6}$, then $D$ contains every orientation of a cycle on $n$ vertices.
\end{theorem}

Then in 2008, Plantholt and Tipnis improved upon Grant's result by showing that if $D$ is a directed graph on $2n$ vertices and $\delta^0(D)\geq \frac{4}{3}n$, then $D$ has an ADHC \cite{PT} (note that this is for all $n$).  Finally in 2011, Busch, Jacobson, Morris, Plantholt, Tipnis improved upon all the results for ADHC's by showing that if $D$ is a directed graph on $2n$ vertices and $\delta^0(D)\geq n+\frac{14}{3}\sqrt{n}$, then $D$ has an ADHC \cite{BJMPT}.

The main goal of this paper is to determine, for sufficiently large $n$, the exact semi-degree threshold for an ADHC. However, we actually prove something stronger which in effect shows that there are only two counterexamples to Grant's conjecture. 

\begin{theorem}\label{main}
Let $D$ be a directed graph on $2n$ vertices.  If $n$ is sufficiently large and $\delta^0(D)\geq n$, then $D$ contains an anti-directed Hamiltonian cycle unless $D$ is isomorphic to $F_{2n}^1$ or $F_{2n}^2$.
\end{theorem}

Since $\delta^0(F_{2n}^1)=\delta^0(F^2_{2n})=n$, we obtain the following corollary.

\begin{corollary}\label{cor:main}
Let $D$ be a directed graph on $2n$ vertices.  If $n$ is sufficiently large and $\delta^0(D)\geq n+1$, then $D$ contains an anti-directed Hamiltonian cycle.
\end{corollary}

Since we have determined the semi-degree threshold for ADHC's, we go back and modify the original conjecture that Grant hinted at.

\begin{conjecture}\label{totdeg}
Let $D$ be a directed graph on $2n$ vertices.  If $\delta(D)\geq 2n+1$, then $D$ contains an anti-directed Hamiltonian cycle.
\end{conjecture}

An \emph{anti-directed $2$-factor} on $n$ vertices is a directed graph in which the underlying graph forms a $2$-factor and no pair of consecutive edges forms a directed path (once again note that $n$ must be even for an anti-directed $2$-factor to exist).  Diwan, Frye, Plantholt, and Tipnis conjectured that if $D$ is a directed graph on $2n\geq 8$ vertices and $\delta^0(D)\geq n$, then $D$ contains an anti-directed $2$-factor \cite{DFPT}.  Since it can be easily shown that $F^1_n$ and $F^2_n$ each contain an anti-directed $2$-factor with two cycles we also obtain the following corollary of Theorem \ref{main}, which implies their conjecture for sufficiently large $n$.
%Slight modifications to the proof of Theorem \ref{main} 
%(we will point out the needed modifications at the appropriate places in the paper) 
%gives the following result which implies their conjecture for sufficiently large $n$.

\begin{corollary}\label{2factor}
Let $D$ be a directed graph on $2n$ vertices.  If $n$ is sufficiently large and $\delta^0(D)\geq n$, then $D$ contains an anti-directed $2$-factor with at most two cycles.
\end{corollary}

%Let $L$ be the (undirected) graph on vertex set $V(L)=\{u_0,u_1,\dots, u_{n-1},v_0,v_1,\dots, v_{n-1}\}$ such that ${u_i, v_j}\in E(L)$ if and only if $|i-j| \bmod{n} \leq 1$.  

Let ${L}_n$ be the graph on vertex set $\{u_1,\dots, u_{n},v_1,\dots, v_{n}\}$ such that $\{u_i, v_j\}\in E({L}_n)$ if and only if $|i-j| \leq 1$.  We call ${L}_n$ a \emph{ladder} and note that ${L}_n$ contains every bipartite $2$-factor on $2n$ vertices.  Let $\vec{L}_n$ be the directed graph obtained from $L_n$ by orienting every edge $\{u_i, v_j\}$ from $u_i$ to $v_j$.  We call $\vec{L}_n$ an \emph{anti-directed ladder} and note that $\vec{L}_n$ contains every anti-directed $2$-factor on $2n$ vertices.

Czygrinow and Kierstead determined the minimum degree threshold for a balanced bipartite graph to contain a spanning ladder.

\begin{theorem}[Czygrinow, Kierstead \cite{CK}]
  \label{CK:ladder}
  There exists $n_0$ such that if $G$ is a balanced bipartite graph on $2n\geq 2n_0$ vertices with $\delta(G)\geq \frac{n}{2}+1$, then $L_n\subseteq G$.
\end{theorem}

We make the following conjecture which would strengthen Corollary \ref{cor:main}.

\begin{conjecture}\label{conj:ladder}
Let $D$ be a directed graph on $2n$ vertices.  If $n$ is sufficiently large and $\delta^0(D)\geq n+1$, then $\vec{L}_n\subseteq D$.  In particular $D$ contains every possible anti-directed $2$-factor.
\end{conjecture}

We note that Conjecture \ref{conj:ladder} holds asymptotically.

\begin{observation}\label{asymptotic}
For all $\ep>0$, there exists $n_0$ such that if $D$ is a directed graph on $2n\geq 2n_0$ vertices with $\delta^0(D)\geq (1+\ep)n$, then $\vec{L}_n\subseteq D$.  
\end{observation}

\begin{proof}
Let $X_1,X_2$ be a random balanced bipartition of $V(D)$.  We expect $\delta^+(x, X_2),\delta^-(x, X_1)\geq \frac{1}{2}(1+\ep)n$ for all $x\in X_1\cup X_2$, so by Chernoffs inequality there exists a such a partition $X_1, X_2$ which satisfies $\delta^+(X_1, X_2)\geq \frac{n}{2}+1$ and $\delta^-(X_2, X_1)\geq \frac{n}{2}+1$.  Let $G$ be an $X_1,X_2$-bipartite graph such that $\{u,v\}\in E(G)$ if and only if $u\in X_1$, $v\in X_2$ and $(u,v)\in \vec{E}_D(X_1, X_2)$.  Note that $G$ is a balanced bipartite graph on $2n$ vertices with $\delta(G)\geq \frac{n}{2}+1$ and thus by Theorem \ref{CK:ladder}, $G$ contains a spanning ladder $L_n$ which corresponds to a spanning anti-directed ladder $\vec{L}_n$ in $D$.
\end{proof}

\section{Overview}

Note that Observation \ref{asymptotic} also implies that Theorem \ref{main} holds asymptotically.  To get the exact result, we use the now common stability technique where we split the proof into two cases depending on whether $D$ is ``close" to an extremal configuration or not (see Figure \ref{GenExtremal}).  If $D$ is close to an extremal configuration, then we use some ad-hoc techniques which rely on the exact minimum semi-degree condition and if $D$ is not close to an extremal configuration then we use the recent absorbing method of R\"odl, Ruci\'nski, and Szemer\'edi (as opposed to the regularity/blow-up method).  
%The absorbing method requires that $n$ be large, but nowhere near as large as required by the regularity lemma.

To formally say what we mean by ``close" to an extremal configuration we need the following definition.

\begin{definition}
Let $D$ be a directed graph on $2n$ vertices.  We say $D$ is $\alpha$-extremal if there exists $A, B\subseteq V(D)$ such that $(1-\alpha)n\leq |A|,|B|\leq (1+\alpha)n$ and $\Delta^+(A,B)\leq \alpha n$ and $\Delta^-(B, A)\leq \alpha n$.
\end{definition}

This definition is more restrictive than simply bounding the number of edges, thus it will help make the extremal case less messy.  However, a non-extremal set still has many edges from $A$ to $B$.

\begin{observation}\label{edgeobs}
Let $0<\alpha\ll 1$.  Suppose $D$ is not $\alpha$-extremal, then for $A, B\subseteq V(D)$ with $(1-\alpha/2)n\leq |A|,|B|\leq (1+\alpha/2)n$, we have $\vec{e}(A, B)\geq \frac{\alpha^2}{2} n^2$.
\end{observation}

\begin{proof}

Let $A, B\subseteq V(D)$ with $(1-\alpha/2)n\leq |A|,|B|\leq (1+\alpha/2)n$.  Since $D$ is not $\alpha$-extremal, there is some vertex $v\in A$ with $\deg^+(v, B)\geq \alpha n$ or $v\in B$ with $\deg^-(v, A)\geq \alpha n$.  Either way, we get at least $\alpha n$ edges.  Now delete $v$, and apply the argument again to get another $\alpha n$ edges.  We may repeat this until $|A|$ or $|B|$ drops below $(1-\alpha)n$, i.e. for at least $\frac{\alpha}{2} n$ steps.  This gives us at least $\frac{\alpha^2}{2}n^2$ edges in total.

\end{proof}

Finally, we make two more observations which will be useful when working with non-extremal graphs.

\begin{observation}\label{robust}
  % TNM change \lambda \ll \alpha to \lambda \le \alpha
Let $0<\lambda \le \alpha\ll 1$ and let $D$ be a directed graph on $n$ vertices.  If $D$ is not $\alpha$-extremal and $X\subseteq V(D)$ with $|X|\leq \lambda n$, then $D'=D-X$ is not $(\alpha-\lambda)$-extremal.
  % TNM
\end{observation}

\begin{proof}
Let $A', B'\subseteq V(D')\subseteq V(D)$ with $(1-\alpha+\lambda)|D'|\leq |A'|,|B'|\leq (1+\alpha-\lambda)|D'|$.  Note that 
\begin{equation*}
(1-\alpha)n\leq (1-\alpha+\lambda)(1-\lambda)n\leq (1-\alpha+\lambda)|D'|\leq |A'|,|B'|\leq (1+\alpha-\lambda)|D'|\leq (1+\alpha)n
\end{equation*}
thus there exists $v\in A'$ such that $\deg^+(v, B')\geq \alpha n\geq (\alpha-\lambda)|D'|$ or $v\in B'$ such that $\deg^-(v, A')\geq \alpha n\geq (\alpha-\lambda)|D'|$.
\end{proof}

\begin{lemma}\label{maxcut}
  % TNM and proper and corrected a typo in the statement of the proof
Let $X, Y\subseteq V(D)$. If $\vec{e}(X, Y)\geq c|X||Y|$, then there exists 
\begin{enumerate}
\item $X'\subseteq X$, $Y'\subseteq Y$ such that $X'\cap Y'=\emptyset$ and
  $\delta^+(X', Y')\geq \frac{c}{8}|Y|, \delta^-(Y',X')\geq \frac{c}{8}|X|$ and

\item a proper anti-directed path in $D[X\cup Y]$ on at least $\frac{c}{4}\cdot\min\{|X|,|Y|\}$ vertices.
\end{enumerate}
\end{lemma}

\begin{proof}
\begin{enumerate}
\item 
  % TNM a few small corrections to the proof
  Let $X^* = X \setminus Y$ and $Y^* = Y \setminus X$.
  Delete all edges not in $\vec{E}(X,Y)$.
  Choose a partition
  $\{X'', Y''\}$ of $X\cap Y$ which maximizes $\vec{e}(X^* \cup X'', Y^* \cup
  Y'')$ and set $X_0=X^* \cup X''$ and $Y_0=Y^* \cup Y''$.
  Note that $\vec{e}(X_0) + \vec{e}(Y_0) + \vec{e}(X_0, Y_0)+\vec{e}(Y_0, X_0) = \vec{e}(X,Y)$.
  We have that 
  \begin{equation*}
    \vec{e}(X_0) =
    \sum_{v \in X_0}\deg^+(v, X_0) =
    \sum_{v \in X''}\deg^-(v, X_0) \leq 
    \sum_{v \in X''}\deg^+(v, Y_0) \leq
    \vec{e}(X_0, Y_0)
  \end{equation*}
  where the inequality holds since if $\deg^-(v, X_0)>\deg^+(v, Y_0)$ for some
  $v\in X''$, then we could move $v$ to $Y''$ and increase the number of edges across the partition.  Similarly, $\vec{e}(X_0, Y_0)\geq \vec{e}(Y_0)$.  Thus $\vec{e}(X_0, Y_0)\geq \frac{1}{4}\vec{e}(X, Y)\geq \frac{c}{4} |X||Y|$.

If there exists $v\in X_0$ such that $\deg^+(v, Y_0)<\frac{c}{8} |Y|$ or $v\in Y_0$ such that $\deg^-(v, X_0)<\frac{c}{8} |X|$, then delete $v$ and set $X_1=X_0\setminus\{v\}$ and $Y_1=Y_0\setminus\{v\}$.  Repeat this process until there no vertices left to delete.  This process must end with a non-empty graph because fewer than $|X|\frac{c}{8}|Y|+|Y|\frac{c}{8}|X|=\frac{c}{4}|X||Y|$ edges are deleted in this process.  Finally, let $X'$ and $Y'$ be the sets of vertices which remain after the process ends.

\item Apply Lemma \ref{maxcut}.(i) to obtain sets $X'\subseteq X$,
  $Y'\subseteq Y$ such that $X'\cap Y'=\emptyset$ and $\delta^+(X', Y')\geq
  \frac{c}{8}|Y|$ and $\delta^-(Y', X')\geq \frac{c}{8}|Y|$.  Let $G$ be an
  auxiliary bipartite graph on $X', Y'$ with $E(G)=\{\{x,y\}:(x,y)\in
  \vec{E}(X', Y')\}$.  Note that $\delta(G)\geq \frac{c}{8}\min\{|X|,|Y|\}$
  and thus $G$ contains a path on at least $2\delta(G)\geq
  \frac{c}{4}\cdot\min\{|X|,|Y|\}$ vertices, which starts in $X$.
  This path contains a proper anti-directed path in $D$
  on at least $\frac{c}{4} \cdot \min\{|X|,|Y|\}$ vertices.
\end{enumerate}
\end{proof}

\section{Non-extremal Case}\label{nonextreme}

In this section we will prove that if $D$ satisfies the conditions of Theorem \ref{main} and $D$ is not $\alpha$-extremal, then $D$ has an ADHC.  We actually prove a stronger statement which in some sense shows that the extremal condition is ``stable," i.e. graphs which do not satisfy the extremal condition do not require the tight minimum semi-degree condition.

\begin{theorem}
  \label{nonex:anti}
  For any $\ec \in (0, 1/32)$ there exists $\ep > 0$ and 
  $n_0$ such if $D = (V,E)$ is a 
  directed graph on $2n \ge 2n_0$ vertices,
  $D$ is not $\ec$-extremal and
  $\msd(D) \ge (1 - \varepsilon)n$, then
  $D$ contains an anti-directed Hamiltonian cycle.
\end{theorem}

\begin{lemma}\label{absorbpath}
  % TNM added proper to the statement of the lemma
 For all $0<\epsilon\ll \beta\ll \lambda \ll\alpha\ll 1$ there exists $n_0$
 such that if $n\geq n_0$, $D$ is a directed graph on $2n$ vertices,
 $\delta^0(D)\geq (1-\ep)n$, and $D$ is not $\alpha$-extremal, then there
 exists a proper anti-directed path $P^*$  with $|P^*|\leq \lambda n$ such
 that for all $W\subseteq V(D)\setminus V(P^*)$ with $2w:=|W|\leq \beta n$,
 $D[V(P^*)\cup W]$ contains a spanning proper anti-directed path with the same
 endpoints as $P^*$.
% For all $0<\epsilon\ll \beta\ll \lambda\ll \alpha\ll 1$ there exists $n_0$ such that if $n\geq n_0$, $D$ is a directed graph on $2n$ vertices, $\delta^0(D)\geq (1-\ep)n$, and $D$ is not $\alpha$-extremal, then there exists an anti-directed path $P^*$ on an even number of vertices with $|P^*|\leq \lambda n$ such that for all $W\subseteq V(D)\setminus V(P^*)$ with $2w:=|W|\leq \beta n$, $D[V(P^*)\cup W]$ contains a spanning anti-directed path with the same endpoints as $P^*$ and the edges incident with the endpoints have the same orientation as in $P^*$.
%TNM
\end{lemma}

\begin{lemma}\label{longpath}
For all $0<\epsilon\ll \beta\ll \lambda \ll \alpha\ll 1$ there exists $n_0$
such that if $n\geq n_0$, $D$ is a directed graph on $2n$ vertices,
$\delta^0(D)\geq (1-\ep)n$, $D$ is not $\alpha$-extremal, and $P^*$ is a
proper anti-directed path with $|P^*|\leq \lambda n$, then $D$ contains an anti-directed cycle on at least $(2-\beta)n$ vertices which contains $P^*$ as a segment.
%For all $0<\epsilon\ll \beta\ll \lambda\ll \alpha\ll 1$ there exists $n_0$ such that if $n\geq n_0$, $D$ is a directed graph on $2n$ vertices, $\delta^0(D)\geq (1-\ep)n$, $D$ is not $\alpha$-extremal, and $P^*$ is an anti-directed path on an even number of vertices with $|P^*|\leq \lambda n$, then $D$ contains an anti-directed cycle on at least $(2-\beta)n$ vertices which contains $P^*$ as a segment.
\end{lemma}

First we use Lemma \ref{absorbpath} and Lemma \ref{longpath} to prove Theorem \ref{nonex:anti}.

\begin{proof}
Let $\alpha\in (0, 1/32)$ and choose $0<\epsilon\ll \beta\ll \lambda\ll \sigma
\ll \alpha$.  Let $n_0$ be large enough for Lemma \ref{absorbpath} and Lemma \ref{longpath}.  Let $D$ be a directed graph on $2n$ vertices with $\delta^0(D)\geq (1-\ep)n$.  Apply Lemma \ref{absorbpath} to obtain an anti-directed path $P^*$ having the stated property. Now apply Lemma \ref{longpath} to obtain an anti-directed cycle $C^*$ which contains $P^*$ as a segment.  Let $W=D-C^*$ and note that since $C^*$ is an anti-directed cycle, $|C^*|$ is even which implies $|W|$ is even, since $|D|$ is even.  Finally apply the property of $P^*$ to the set $W$ to obtain an ADHC in $D$.
\end{proof}

\subsection{Absorbing}

To prove Lemma \ref{absorbpath} we will use the following more general statement.

\begin{lemma}\label{generalabsorbing}
  Let $m,d \in \N$, $a > 0$,
  $b \in (0, \frac{a}{2d})$ and
  $c \in 
  \left(0, 2 b \left(\frac{a}{2d} - b\right) \right)$.
  There exists $n_0$ such that 
  when $V$ is a set of order $n \ge n_0$ the following holds.
  For every $S \in \binom{V}{m}$, let $f(S)$ be a subset of $V^d$.
  Call $T \in V^d$ a \emph{good tuple} if $T \in f(S)$ 
  for some $S \in \binom{V}{m}$.
  If $|f(S)| \ge a n^d$ for every $S \in \binom{V}{m}$ then
  there exists a set $\mathcal{F}$ of at most $b n/d$
  good tuples such that
  $|f(S) \cap \mathcal{F}| \ge c n$ for every $S \in \binom{V}{m}$
  and the images of distinct elements of $\mathcal{F}$ are disjoint.
  \label{lem:absorbing}
\end{lemma}
\begin{proof}
  Pick $\varepsilon > 0$ so that  
  \begin{equation*}
    (1 + a)\varepsilon <  \frac{a b}{d} - 2b^2 - c.
  \end{equation*}
  Let $b' := \frac{b}{d}$,
  $p := b' - \varepsilon$ and 
  $c' := c + (d^2 + 1)p^2$.
  Let $\mathcal{F}'$ be a random subset of $V^d$ where each $T \in V^d$
  is selected independently with probability $p n^{1 - d}$.
  Let 
  \begin{equation*}
    \mathcal{O} := 
    \left\{ 
      \{T, T'\} \in \binom{V^d}{2} : \im(T) \cap \im(T') \neq  \emptyset 
    \right\}
  \end{equation*}
  and $\mathcal{O}_{\mathcal{F}'} := \mathcal{O} \cap \binom{\mathcal{F}'}{2}$.

  We only need to show that, for sufficiently large $n_0$,
  with positive probability 
  $|\mathcal{O}_{\mathcal{F}'}| < (d^2 + 1)p^2 n$,
  $|\mathcal{F}'| < b' n$ and
  $|f(S) \cap \mathcal{F}'| > c' n$ for every $S \in \binom{V}{m}$.
  We can then remove at most $(d^2 + 1)p^2 n$ tuples
  from such a set $\mathcal{F}'$ so that the images of the remaining
  tuples are disjoint.
  After also removing every 
  $T \in \mathcal{F}'$ for which there is no $S \in \binom{V}{m}$
  for which $f(S) = T$, 
  the resulting set $\mathcal{F}$ will satisfy the conditions of the lemma.

  Clearly,  
  \begin{equation*}
    |\mathcal{O}| \le n \cdot d^2 \cdot n^{2d-2}=
    d^2 n^{2d - 1},
  \end{equation*}
  and for any $\{T, T'\} \in \binom{V^d}{2}$,
  $\Pr(\{T, T'\} \subseteq \mathcal{F}') = p^2 n^{2 - 2d}$.
  Therefore, by the linearity of expectation,
  $\E[|\mathcal{O}_{\mathcal{F}'}|] < d^2 p^2 n$.
  So, by Markov inequality, 
  \begin{equation*}
    \Pr\left( 
    \left|\mathcal{O}_{\mathcal{F}'} \right| \ge (d^2 + 1) p^2 n 
    \right) \le 
    \frac{d^2}{d^2 + 1}.
  \end{equation*}
  Note that $\E[|\mathcal{F}'|] = p n$ and 
  $p n \ge \E[|f(S) \cap \mathcal{F}'|] \ge a p n$ for every $S \in \binom{V}{m}$.
  Therefore, by the Chernoff inequality,
  % TNM - removed precise version of Chernoff's inequality
  %  Therefore, by Theorem~\ref{chernoff},
  %  \begin{equation*}
  %    \Pr(|\mathcal{F}'| \ge b' n) \le 
  %    \exp\left(\frac{-3\varepsilon^2 n}{6p + 2\varepsilon}\right),
  %  \end{equation*}
  $
  \Pr(|\mathcal{F}'| \ge b' n) \le 
  e^{-\varepsilon^2 n/3}
  $
  and, since 
  \begin{equation*}
    a p - c' =
    \frac{a b}{d} - a \varepsilon - 
    (d^2+1)\left(\frac{b}{d} - \varepsilon\right)^2 - c \ge
    \frac{a b}{d} - 2 b^2 - c - a \varepsilon
    > \varepsilon,
  \end{equation*} 
  % TNM - removed precise version of Chernoff's inequality
  %  $\Pr(|\mathcal{F}' \cap f(S)| \le c' n) \le 
  %  \exp\left(-\frac{\varepsilon^2 n}{2p} \right)$
  $\Pr(|\mathcal{F}' \cap f(S)| \le c' n) < e^{-\varepsilon^2 n/3}$
  for every $S \in \binom{V}{m}$.
  Therefore, for sufficiently large $n_0$,
  \begin{equation*}
    \Pr\left( 
    \left|\mathcal{O}_{\mathcal{F}'} \right| \ge (d^2 + 1) p^2
    \right) 
    +
    \Pr(|\mathcal{F}'| \ge b' n) 
    +
    \sum_{S \in \binom{V}{m}} \Pr(|\mathcal{F}' \cap f(S)| \le c' n)
    < 1.
    \qedhere
  \end{equation*} 
\end{proof}

Let $\pairs := V^2 - \{(x,x) : x \in V\}$.
For any $(x,y) \in \pairs$, 
call $(a,b,c,d) \in V^4$ 
an $(x,y)$-\emph{absorber} if 
% TNM added proper here
$abcd$ is a proper anti-directed path and
$axcbyd$ is a proper anti-directed path (see Figure~\ref{fig:absorber}) and call $(a,b) \in V^2$ an $(x,y)$-\emph{connector} if 
%$abcd$ is an anti-directed path with $(a,b)\in \vec{E}(D)$ and
%$axcbyd$ is an anti-directed path with $(a,b)\in \vec{E}(D)$ (see Figure~\ref{fig:absorber}) and call $(a,b) \in V^2$ an $(x,y)$-\emph{connector} if 
$xaby$ is an anti-directed path where $(a,b)$ is an edge (note that specifying one edge dictates the directions of all the other edges).

Note that if $(x',x), (y, y')\in \vec{E}(D)$ and
$(a,b)$ is an $(x,y)$-connector disjoint from $\{x',y'\}$ 
then $x'xabyy'$ is an anti-directed path.

For all $(x,y)\in \pairs$, let $f_{\text{abs}}(x,y)=\{T \in V^4 : T \text{ is an $(x,y)$-absorber}\}$ and $f_{\text{con}}(x,y)=\{T \in V^2 : T \text{ is an $(x,y)$-connector}\}$.

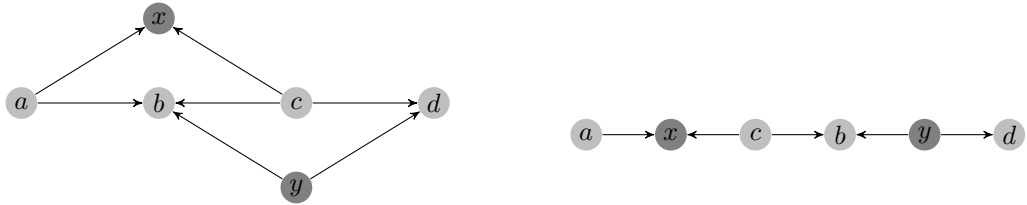
\begin{figure}[h]
  \begin{tikzpicture}[
      >=stealth',
      pil/.style={
	->,
	thick,
	shorten <=2pt,
	shorten >=2pt,
      }
      scale=0.5,
      avertex/.style={circle,fill=gray,minimum size=12pt},
      pvertex/.style={circle,fill=lightgray,minimum size=12pt},
      font=\small,minimum size=10pt,
    inner sep=1pt]
    \matrix[row sep=20pt,column sep=20pt] {
      & & & \node [avertex] (x) {$x$}; & & & & & \\
      & \node [pvertex] (a) {$a$}; & 
      & \node [pvertex] (b) {$b$}; &
      & \node [pvertex] (c) {$c$}; & 
      & \node [pvertex] (d) {$d$}; \\
      & & & & & \node [avertex] (y) {$y$}; & & \\
    };
    \draw [->] (a) to (b);
    \draw [<-] (b) to (c);
    \draw [->] (c) to (d);
    \draw [->] (a) to (x);
    \draw [->] (c) to (x);
    \draw [<-] (b) to (y);
    \draw [<-] (d) to (y);
  \end{tikzpicture}
  \begin{tikzpicture}[
      >=stealth',
      pil/.style={
	->,
	thick,
	shorten <=2pt,
	shorten >=2pt,
      }
      scale=0.5,
      avertex/.style={circle,fill=gray,minimum size=12pt},
      pvertex/.style={circle,fill=lightgray,minimum size=12pt},
      font=\small,minimum size=10pt,
    inner sep=1pt]
    \matrix[row sep=20pt,column sep=20pt] {
      & & & & & & & & \\
      & \node [pvertex] (a) {$a$};
      & \node [avertex] (x) {$x$};
      & \node [pvertex] (c) {$c$};
      & \node [pvertex] (b) {$b$};
      & \node [avertex] (y) {$y$};
      & \node [pvertex] (d) {$d$}; \\
      & & & & & & & \\
    };
    \draw [->] (a) to (x);
    \draw [<-] (x) to (c);
    \draw [->] (c) to (b);
    \draw [<-] (b) to (y);
    \draw [->] (y) to (d);
  \end{tikzpicture}
  \caption{$(a, b, c, d)$ is an $(x,y)$-absorber}\label{fig:absorber}
\end{figure}

\begin{claim}\label{manyabsorbers} Let $D$ satisfy the conditions of Lemma \ref{absorbpath}.  For all $(x,y)\in \pairs$ we have 
\begin{enumerate}
\item $|f_{abs}(x,y)|\geq \alpha^{12} n^4$ and 

\item $|f_{\text{con}}(x,y)|\geq \alpha^3n^2$.
\end{enumerate}
\end{claim}

\begin{proof}
Let $(x,y)\in \pairs$ and let $A=N^-(x)$ and $B=N^+(y)$.  

\begin{enumerate}

    % TNM added +1 to inequalities since we need to avoid b and c
    % when selecting a and d.
\item By Observation \ref{edgeobs} and Lemma \ref{maxcut}, there exists
  $A'\subseteq A$ and $B'\subseteq B$ such that $A'\cap B'=\emptyset$ and
  $\delta^+(A', B'), \delta^-(B', A')\geq \frac{\alpha^2}{16}(1-\ep)n\ge 
  \alpha^3n + 1$.  For all $(c,b)\in \vec{E}(A',B')$, we have
  $|N^-(b)\cap A'|\geq \alpha^3n + 1$ and $|N^+(c)\cap B'|\geq \alpha^3n + 1$.  So
  there are more than $(\alpha^3n)^2$ choices for $(b,c)$, $\alpha^3n$ choices for $a$ and $\alpha^3n$ choices for $d$, i.e. $|f_{abs}(x,y)|\geq \alpha^{12} n^4$.

\item Similarly, by Observation \ref{edgeobs}, we have 
  % TNM - removed (1-\ep)^2
  $\vec{e}(A, B)\geq \frac{\alpha^2}{2}n^2\geq \alpha^3n^2$, each of which is a connector.
 
\end{enumerate}

\end{proof}

\begin{claim}[Connecting-Reservoir]\label{reservoir}
For all $0<\gamma\ll \alpha$ and $D' \subseteq D$ such that $|D'| \ge (2
- \lambda)n$, there exists a set of pairwise disjoint ordered pairs $\mathcal{R}$ such that
if $R=\cup_{(a,b)\in \mathcal{R}}\{a,b\}$, then
$R\subseteq V(D')$, $|R|\leq \gamma n$
and for all distinct $x,y\in V(D)$, 
$|f_{con}(x,y)\cap \mathcal{R}| \geq \gamma^2 n$.
\end{claim}

\begin{proof}
For every $(x,y) \in \mathcal{P}$
\begin{equation*}
  |\{(a,b) \in f_{con}(x,y) : a,b \in V(D')\}| \geq 
  |f_{con}(x,y)| - |D - D'|n \ge \alpha^3n^2/2.
\end{equation*}
Therefore, we can apply Lemma \ref{generalabsorbing}
to obtain a set $\mathcal{R}$ of disjoint good ordered pairs such that
$|\mathcal{R}|\leq \gamma n/2$ and $|f_{con}(x,y)\cap \mathcal{R}|\geq
\gamma\alpha^3n/4 - 2\gamma^2n \ge \gamma^2n$ 
and $\mathcal{R} \subseteq V(D')^2$.
\end{proof}

Now we prove Lemma \ref{absorbpath}.

\begin{proof}
  % TNM - I don't think we need Claim \ref{reservoir} yet - left old proof
  % comment out in case I am wrong.
%Start by applying Claim \ref{reservoir} with $\gamma=\alpha^{13}$ to obtain $R$.  Let $D'=D-R$ and let $V'=V(D')$.

%Since $|f_{abs}(x,y)\cap \mathcal{P}(V')|\geq (\alpha^{12}-\alpha^{13}) n^4\geq \frac{\alpha^{12}}{2}n^4$ we apply Lemma \ref{generalabsorbing} to $D-R$ obtain a set $\mathcal{A}$ of disjoint good $4$-tuples $\{A_1, \dots, A_\ell\}$ such that $|\mathcal{A}|\leq \lambda n/8$ and $|f_{abs}(x,y)\cap \mathcal{A}|\geq \lambda\frac{\alpha^{12}}{2} n\geq \lambda^2n$.  Let $A=\cup_{(a,b,c,d)\in \mathcal{A}}\{a,b,c,d\}$ and note that $|A|\leq \lambda n/2$.
Since $|f_{abs}(x,y)\cap \mathcal{P}(V')|\geq \alpha^{12} n^4$ we apply Lemma
\ref{generalabsorbing} to $D$ obtain a set $\mathcal{A}$ of disjoint good
$4$-tuples $\{A_1, \dots, A_\ell\}$ such that $|\mathcal{A}|\leq \lambda n/8$
and $|f_{abs}(x,y)\cap \mathcal{A}|\geq \lambda \alpha^{12}n/8 -
2(\lambda/2)^2n \geq \lambda^2n$.  Let $A=\cup_{(a,b,c,d)\in \mathcal{A}}\{a,b,c,d\}$ and note that $|A|\leq \lambda n/2$.

Let $(a_i, b_i, c_i, d_i) := A_i$ for every $i \in [l]$, so
$a_ib_ic_id_i$ is a proper
ADP.  Note that there are less than $|A|n$ ordered pairs that contain a vertex
from $A$, so since $\lambda\ll \alpha$, we can greedily choose vertex 
disjoint $(x_i, y_i)\in f_{con}(d_i, a_{i+1})$ for each $i \in [l-1]$ such that 
$x_i,y_i \notin A$.  Set $P^* :=
A_1x_1y_1A_2x_2y_2A_2 \dotsc A_{l-1}x_{l-1}y_{l-1}A_l$ and note that $|P^*|
\leq \lambda n$ and $|P^*|$ is a proper ADP.  
%Let $(a_i, b_i, c_i, d_i) := A_i$ for every $i \in [l]$ so that $(a_i, b_i), (c_i, b_i), (c_i,d_i)\in \vec{E}(D)$.  Since $\lambda\ll \alpha$, we can greedily choose vertex disjoint $(x_i, y_i)\in f_{con}(d_i, a_{i+1})$ for each $i \in [l-1]$.  Set $P^* := A_1x_1y_1A_2x_2y_2A_2 \dotsc A_{l-1}x_{l-1}y_{l-1}A_l$ and note that $|P^*|\leq \lambda n$ and $|P^*|$ is even.  

To see that $P^*$ has the desired property, let $W\subseteq V\setminus V(P^*)$
such that $2w=|W|\leq \beta n$.  Arbitrarily partition $W$ into pairs and
since $\beta\ll \lambda$, we can greedily match the disjoint pairs from $W$
with $4$-tuples in $\mathcal{A}$.  By the way we have defined an
$(x,y)$-absorber, $D[V(P')\cup W]$ contains a spanning proper anti-directed path starting with an out-edge from $a_1$ and ending with an in-edge to $d_\ell$.
%To see that $P^*$ has the desired property, let $W\subseteq V\setminus V(P^*)$ such that $2w=|W|\leq \beta n$.  Arbitrarily partition $W$ into pairs and since $\beta\ll \lambda$, we can greedily match the disjoint pairs from $W$ with $4$-tuples in $\mathcal{A}$.  By the way we have defined an $(x,y)$-absorber, $D[V(P')\cup W]$ contains a spanning anti-directed path starting with an out-edge from $a_1$ and ending with an in-edge to $d_\ell$.

\end{proof}

\subsection{Covering}

The main challenge in the proof of Lemma \ref{longpath} is to show that if a maximum length anti-directed path is not long enough, then we can build a constant number of vertex disjoint anti-directed paths whose total length is sufficiently larger.

%TNM changed statement added proper and r \le 5 
% and changed $(1 - \beta)n$ to $(2 - \beta)n$
\begin{claim}
 \label{clm:covering} 

Under the conditions of Lemma \ref{longpath}, suppose $P^*$ is a proper anti-directed path with $|P^*|\leq \lambda n$.  For all $R\subseteq V(D-P^*)$ with $|R|\leq \beta^2n$, if $P$ is a proper anti-directed path in $D-R$ with beginning segment $P^*$ such that $|P| < (2 - \beta)n$, then there exist disjoint proper anti-directed paths
  $Q_1, \dotsc, Q_r \subseteq D-R$, such that $r \le 6$,
  $Q_1$ contains $P^*$ as an initial segment and
  \begin{equation*}
    |Q_1| + \dotsm |Q_r| \ge |P| + \ep \ceiling{\frac{1}{4}\log n}.
  \end{equation*}
\end{claim}

First we show how this implies Lemma \ref{longpath}.

\begin{proof}
%Assume that $n$ is large enough so that
%\begin{align}
%  \label{eq:large_n_1}
%  n &\ge \ec^{-28} \cdot 2m \cdot 2^{2m} \ge 
%  2 \ec^{-28}(\frac{1}{4} \log n)(\log n)^2 \text{ and }\\
%  \label{eq:large_n_2}
%  m &> 10 \ec^{-30}
%\end{align}
Let $n$ be large enough so that we can apply Claim \ref{reservoir} and so that if 
 $m:=\ceiling{\frac{1}{4}\log n}$, then 
  \begin{equation}\label{eq:large_n}
  n\geq \frac{4m2^{2m}}{\ep^2\beta}
  \text{ and }
  m > 10\beta^{-4}\varepsilon^{-1}.
  \end{equation}

Let $P^*$ be a proper anti-directed path with $|P^*|\leq \lambda n$.  Let
$D'=D-P^*$.  Now apply Claim \ref{reservoir} to $D'$ with $\gamma=\beta^2$ to get
$\mathcal{R}$ and $R$ such that $|f_{con}(x,y)\cap \mathcal{R}|\geq \beta^4n$
for every $(x, y) \in \mathcal{P}$ and $|R| \le \beta^2 n$.
%and note that $\msd(D')\geq (1-\ep-\lambda)n$ and $D'$ is not $(\alpha-\lambda)$-extremal by Observation \ref{robust}.  So by Claim \ref{manyabsorbers}, $|f_{con}(x,y)|\geq \frac{(\alpha-\lambda)^2}{2}(1-\ep-\lambda)^2n^2\geq \alpha^3n^2$.  Now apply Claim \ref{reservoir} with $\gamma=\beta^2$ to get $R$ such that $|f_{con}(x,y)\cap \mathcal{R}|\geq \beta^4n$.

Let $P$ be a maximum length proper anti-directed path in $D-R$ that begins
with $P^*$.  If $|P|<(2-\beta)n$, then we apply Claim~\ref{clm:covering}.
Now connect $Q_1, \dotsc, Q_r$ into a longer path using at most $5$ pairs from
$\cpairs$.  Delete these vertices from $R$ and reset $\cpairs$.  We may repeat
this process as long as there are sufficiently many pairs remaining in
$\cpairs$.  On each step, $|f_{con}(x,y)\cap \mathcal{R}|$ may be reduced by
at most $5$.  However, in less than $\frac{2n}{\ep m}$ steps, we
will have a path of length greater than $(2-\beta)n$ in which
case we would be done.  By \eqref{eq:large_n}, $5 \cdot \frac{2n}{\ep m} <
\beta^4n$, so we can repeat the process sufficiently many times.  Once we have a path $P$ with $|P|\geq (2-\beta)n$, we use one more pair from $\cpairs$ to connect the endpoints of $P$ to form an anti-directed cycle $C$, which is possible since $|P|$ is even.  Note that $C$ contains $P^*$ as a segment by construction.

%
%By Claim~\ref{clm:connecting_reservoir}, we
%can repeat this procedure using distinct
%elements from $\cpairs$ until $|P| \ge (1 - \ec^{22})2n$
%since $\ec^{26}n > 5(2n \ec^{-4} m^{-1})$ by \eqref{eq:large_n_2}.
%We then use one pair from $\cpairs$ to connect the
%ends of $P$ to form an anti-directed cycle $C$.
%Let $\mathcal{T}$ be an arbitrary partition of the less
%than $2 \ec^{21}n$ vertices in $V(D - C)$ into sets of size $2$.
%By Claim~\ref{clm:absorbing_path},
%for every $\{x,y\}$ in $\mathcal{T}$
%there exists 
%a unique $(a_i, b_i, c_i, d_i) \in f(x, y) \cap \aquads$ 
%(here the order $(x,y)$ is arbitrary).
%Replacing $a_ib_ic_id_i$ with $a_ixc_ib_iyd_i$ in $P$ 
%then gives the desired anti-directed Hamiltonian cycle.

\end{proof}

\begin{proof}[Proof of Claim \ref{clm:covering}]
Let $n$ and $m$ be as in \eqref{eq:large_n}. 
Let $P$ be a maximum length proper ADP in $D - R$ containing $P^*$ as an
initial segment. Let $\hat{P}$ be the shortest segment of $P$ immediately
following $P^*$ so that $P':= v_1 \dotsc v_p = P - (P^*\cup \hat{P})$ is a multiple of $2m$; thus $|\hat{P}|<2m$. Set $T := V \setminus (V(P) \cup V(R))$, and 
  $P_i := v_{2m(i-1) + 1} \dotsc v_{2mi}$
  for $i \in [s]$ where $s := \frac{p}{2m}$ (which is an integer by the choice of $\hat{P}$).
  Note that $|P_i| = 2m$ for every $i \in [s]$.
  Assume $|T| > \beta n - |R| > \beta n / 2$.

\begin{claim}
    \label{clm:complete_subgraph}
    Let $c \in (\ep^2-1, 1)$, $d \in [\ep^2, 1+c)$, and 
    $b := \ceiling{(1 + c  - d)m}$.
    If $\rdeg{T}{P_i} \ge (1 + c)m|T|$, then there exists 
    $X_i \subseteq V(P_i)$ and $Y_i \subseteq T$ 
    such that
    $|X_i| = b$,
    $|Y_i| \ge 2m$ and 
    $X_i \subseteq \outn{y}$ 
    for every $y \in Y_i$. In particular, $D[V(P_i)\cup T]$ contains a proper anti-directed path on $2b$ vertices.
  \end{claim}
  
\begin{proof}
  % TNM changed the computation here slightly
Let $T'=\{v\in T:\deg^+(v, P_i)\geq b\}$ and since 
\begin{equation*}
(1+c)m|T| \leq \vec{e}(T, P_i)\leq (|T|-|T'|)(b-1)+|T'|(2m - b + 1)
\leq|T|(1 + c - d)m + |T'|2m
\end{equation*}
we have $|T'|\geq \frac{d}{2}|T|$. 
%which implies $|T'|\geq \frac{(1+\ep)(d-c)m}{(1-c+d)m}|T|\geq  \frac{d}{2}|T|$. 
Together with  \eqref{eq:large_n} this gives us
    \begin{equation*}
      |T'|\geq \frac{d}{2}|T| \ge \ep^2\beta n/2 \ge 2m 2^{2m} > 
      2m \binom{2m}{b}, 
    \end{equation*}
    which by the pigeonhole principle implies that
    there exists 
    $X_i \subseteq V(P_i)$ with $|X_i|= b$ and 
    $Y_i \subseteq T'$ such that
    $|Y_i| \ge 2m$ and
    $X_i \subseteq N_H(y)$ for every $y \in Y_i$.
  \end{proof}

  By Claim~\ref{clm:complete_subgraph}, 
  if $\rdeg{T}{P_i} \ge (1 + \ep)|T|m$ there 
  exists a proper anti-directed path $Q_3$ of length
  \begin{equation*}
    2\ceiling{(1 + \ep - \ep^2)m}>(2 + \ep)m \text{ in } D[T \cup P_i].
  \end{equation*}
  Letting 
  $Q_1 := P^*\hat{P}P_1 \dotsm P_{i-1}$ and
  $Q_2 := P_{i+1} \dotsm P_{q}$ then satisfies the condition
  of the lemma.
  Therefore, we can assume that,
  \begin{equation}
    \label{eq:deg_T_P_i}
    \rdeg{T}{P_i} < (1 + \ep)|T|m
    \text{ for every $i \in [s]$}.
  \end{equation}

  We can also assume that
  \begin{equation}
    \label{eq:deg_T}
    \rdeg{T}{T} < \ep |T|^2.
  \end{equation}
  Otherwise by Lemma \ref{maxcut}.(ii) there exists a proper
  anti-directed path $Q_2$ of length $\frac{\ep}{4}|T| \geq \ep m$ in 
  $D[T]$.  Then $Q_1 := P$ and $Q_2$ satisfy the condition of the
  lemma.

  So \eqref{eq:deg_T} implies that
  \begin{equation}
    \begin{split}
    \label{eq:deg_to_P'}
    \rdeg{T}{P'} & \ge 
    (1 - \ep)n|T| - (|P^*|+|\hat{P}|+|R|)|T| - \rdeg{T}{T} \\
    &\ge (1-\ep-\lambda-\beta^2)n|T|-2m|T| -\ep |T|^2 \ge (1 - 2\lambda)n|T|
  \end{split}
  \end{equation}
  Let $\lambda\ll \sigma\ll \alpha$ and let 
  \begin{equation*}
    I := \{ i \in [s] : \rdeg{T}{P_i} \ge (1 - \sigma)|T|m \}.
  \end{equation*}
  By \eqref{eq:deg_T_P_i} and \eqref{eq:deg_to_P'},
  \begin{equation*}
    (1 - 2\lambda)n|T| \le \rdeg{T}{P'} \le 
    (1 - \sigma)m(s-|I|)|T| + (1 + \ep)m|I||T|\leq (1-\sigma)n|T|+(\sigma+\ep)m|I||T|
  \end{equation*}
  which implies that 
  $m|I| \ge \frac{\sigma - 2\lambda}{\sigma + \ep}n > (1 - \frac{\alpha}{2})n$.
  Also note that $n \ge |P|/2 \ge m|I|$.
  
  For every $i \in I$, let $X_i \subseteq P_i$ 
  and $Y_i \subseteq T$ be the
  sets guaranteed by Claim~\ref{clm:complete_subgraph}
  with $c := -\sigma$, $d := \sigma$ and
  $b := \lceil (1 - 2 \sigma)m \rceil$.
  Let $Z_i := V(P_i) \setminus X_i$ for $i \in [I]$ and
  let $Z := \bigcup_{i \in I} Z_i$.
  Note that $|Z_i| = 2m - b$ for every $i \in I$ so $|Z|=(2m-b)|I|$ and 
  \begin{equation*}
    \left(1+\frac{\alpha}{2}\right)n>(1 + 2 \sigma)n \ge (2m - b)|I| \ge m|I| > \left(1 - \frac{\alpha}{2}\right)n.
  \end{equation*}

  Therefore by Observation \ref{edgeobs}, $\rdeg{Z}{Z} \ge \frac{\alpha^2}{2}
  |Z|^2$. 
  Because
  \begin{equation*}
    \frac{\alpha^2}{2}\leq \frac{\rdeg{Z}{Z}}{|Z|^2} = 
    \frac{1}{|I|^2}\sum_{i \in I}\sum_{j \in I} \frac{\rdeg{Z_i}{Z_j}}{(2m - b)^2},
  \end{equation*}
  there exists $i, j \in I$ such that $\rdeg{Z_i}{Z_j} \ge \ec^2 (2m - b)^2/2$.
  Removing $P_i$ and $P_j$ divides $P$ into 
  three disjoint proper anti-directed paths (note that some of these paths may be empty).
  Label these paths $Q_1$, $Q_2$ and $Q_3$ so that $P^*\hat{P} \subseteq Q_1$.
  By Lemma \ref{maxcut}.(ii) there
  exists a proper anti-directed path $Q_4$ of length at least $(\ec^2/8) (2m - b) \ge (\ec^2/8)m$
  in $D[Z_i \cup Z_j]$.
  By Claim \ref{clm:complete_subgraph}, there also exists a proper 
  anti-directed path 
  $Q_5 \subseteq D[X_i \cup Y_i]$ such
  that $|Q_5| \ge 2(1 - 2\sigma)m$.

  If $i = j$ then $Q_4 \subseteq D[Z_i]$
  and $|Q_1| + |Q_2| + |Q_3| = |P| - 2m$.
  Therefore it is enough to observe that 
  $|Q_4| + |Q_5| \ge 2(1 - 2\sigma)m + (\ec^2/8) m \ge
  2m + \ep m$.

  If $i \neq j$, then 
  $Y_j' := Y_j \setminus V(Q_4)$
  has order at least $2m - b \ge m$.
  So there exists a proper anti-directed path
  $Q_6 \subseteq D[X_j \cup Y_j']$
  such that $|Q_6| \ge 2(1 - 2\sigma)m$.
  Since $|Q_1| + |Q_2| + |Q_3| = |P| - 4m$
  and $|Q_4| + |Q_5| + |Q_6| \ge 4(1 - 2\sigma)m + (\ec^2/8) m \ge
  4m + \ep m$, the proof is complete.
\end{proof}

\section{Extremal Case}

Let $0<\alpha\ll \beta\ll \gamma\ll 1$.  Let $D$ be a directed graph on $2n$ vertices with $\delta^0(D)\geq n$ and suppose that $D$ satisfies the extremal condition with parameter $\alpha$ and that $D$ is not isomorphic to $F_n^1$ or $F_n^2$.  We will first partition $V(D)$ in the preprocessing section, then we will handle the main proof.
In this section we sometime use $uv$ to denote the edge $(u,v)$.

\subsection{Preprocessing}

The point of this section is to make the following statement precise: If
$D$ satisfies the extremal condition, then $D$ is very similar to the digraph in Figure \ref{GenExtremal}.

\begin{proposition}\label{prop:preprocess}
If there exists an $\alpha$-extreme pair of sets $A,B\subseteq V(G)$, then there exists a partition $\{X_1',X_2',Y_1',Y_2',Z\}$ of $V(G)$ such that 
\begin{enumerate}
\item\label{close} $|Z'|\leq 3\alpha^{2/3}n$, 
  $||X_1'|-|X_2'||, ||Y_1'|-|Y_2'||\leq 3\alpha^{2/3}n$ and
\item 
%  \begin{align*}
$ \delta^0(X_{3-i}',X_i'), \delta^-(Y_{3-i}', X_i'), \delta^+(Y_i',
    X_i') \geq |X_i'|-2\alpha^{1/3}n \text{ and } \\
    \delta^0(Y_i',Y_i'),\delta^-(X_i', Y_i'), \delta^+(X_{3-i}', Y_i') \geq
    |Y_1'|-2\alpha^{1/3}n$ for $i = 1,2$.
%    \delta^0(D[Y_i'])&\geq |Y_i'|-2\alpha^{1/3}n.
%\end{align*}

\end{enumerate}

\end{proposition}

\begin{proof}

Let $A, B\subseteq V(D)$ such that $(1-\alpha)n\leq |A|,|B|\leq (1+\alpha)n$, $\Delta^+(A,B)\leq \alpha n$, and $\Delta^-(B,A)\leq \alpha n$.  
We have that 
\begin{align}
  \delta^+(A, \overline{B})&\geq (1-\alpha)n, \text{ and} \label{y1x2}\\
  \delta^-(B, \overline{A})&\geq (1-\alpha)n. \label{y2x2}
\end{align}
Set $\widetilde{X_1}=V\setminus(A\cup B)$, $\widetilde{X_2}=A\cap B$, 
$\widetilde{Y_1}=A\setminus B$,
$\widetilde{Y_2}=B\setminus A$. Note that
$\widetilde{Y_1} \cup \widetilde{X_2} = A$ and
$\widetilde{Y_2} \cup \widetilde{X_2} = B$.
Therefore, $||\widetilde{Y_1}| - |\widetilde{Y_2}|| \le 2\alpha n$, and
$||\widetilde{X_1}| - |\widetilde{X_2}|| \le 2\alpha n$,
because $|\widetilde{X_1}| - |\widetilde{X_2}| = |V| - |A| - |B|$.

Let 
\begin{align*}
  \hat{Y_1} = \{v\in \widetilde{Y_1}: 
  &\deg^-(v, \widetilde{X_2})<|\widetilde{X_2}|-\alpha^{1/3}n \text{ or } 
  \deg^-(v, \widetilde{Y_1})<|\widetilde{Y_1}|-\alpha^{1/3}n\}, \\
  \hat{Y_2} =\{v\in \widetilde{Y_2}: 
  &\deg^+(v, \widetilde{X_2})<|\widetilde{X_2}|-\alpha^{1/3}n \text{ or } 
  \deg^+(v, \widetilde{Y_2})<|\widetilde{Y_2}|-\alpha^{1/3}n\}, \\
  \hat{X_1} =\{v\in \widetilde{X_1}: 
  &\deg^-(v, \widetilde{Y_1})<|\widetilde{Y_1}|-\alpha^{1/3}n \text{ or } 
  \deg^+(v, \widetilde{Y_2})<|\widetilde{Y_2}|-\alpha^{1/3}n \text{ or } \\ 
  &\deg^0(v, \widetilde{X_2})<|\widetilde{X_2}|-\alpha^{1/3}n\},
\end{align*}
$\hat{B} = \hat{Y_1} \cup \hat{X_1}$ and
$\hat{A} = \hat{Y_2} \cup \hat{X_1}$.
Note that $\hat{B} \subseteq \overline{B}$
and $\hat{A} \subseteq \overline{A}$.
Now we show that each of these sets are small.

\begin{claim}\label{smallhats}
  $|\hat{Y_1}|,|\hat{Y_2}|,|\hat{X_1}|\leq 2\alpha^{2/3}n$ and 
  $|\hat{Y_1}|+|\hat{Y_2}|+|\hat{X_1}|\leq 3\alpha^{2/3}n$
\end{claim}

\begin{proof}

By \eqref{y1x2} and the definition of $\hat{X}_1, \hat{Y}_1$, we have 
\begin{equation*}
  |\widetilde{Y_1}\cup \widetilde{X_2}|(1-\alpha)n =
  |A|(1 - \alpha)n \leq 
  \vec{e}(A, \overline{B}) \leq 
  (|\overline{B}|-|\hat{B}|)|A|+|\hat{B}|(|A|-2\alpha^{1/3}n)
\end{equation*}

This implies 
\begin{equation*}
  |\hat{Y}_1\cup\hat{X}_1| =
  |\hat{B}| \leq 
  \frac{|A|(|\overline{B}|-(1-\alpha)n)}{2\alpha^{1/3}n} \leq
  \frac{(1+\alpha)n((1+\alpha)n-(1-\alpha)n)}{2\alpha^{1/3}n} = 
  (1 + \alpha)\alpha^{2/3}n
\end{equation*}

Now using \eqref{y2x2}, the same calculation 
(with the symbol $A$ exchanged with the symbol $B$) gives
that 
$|\hat{Y}_2\cup\hat{X}_1| = |\hat{A}| \leq (1 + \alpha)\alpha^{2/3}n$.

Thus $|\hat{Y_1}|+|\hat{Y_2}|+|\hat{X_1}|\leq 2(1+\alpha)\alpha^{2/3}n\leq 3\alpha^{2/3}n$.
\end{proof}

Let $X_1'= \widetilde{X_1} \setminus \hat{X}_1$, $X_2'= \widetilde{X_2}$,
$Y_i'=\widetilde{Y_i}\setminus
\hat{Y}_i$ for $i=1,2$, and $Z=\hat{X}_1\cup \hat{Y}_1\cup \hat{Y}_2$.  Note
that $|Z|\leq 3\alpha^{2/3}n$ and
$||X_1'|-|X_2'||, ||Y_1'|-|Y_2'||\leq 2\alpha n+2\alpha^{2/3}n<3\alpha^{2/3}n$. 
The required
degree conditions all follow from \eqref{y1x2} and \eqref{y2x2}; the
definitions of $\hat{X}_1$, $\hat{Y}_1$ and $\hat{Y}_2$; 
and Claim \ref{smallhats} .
\end{proof}

\subsection{Preliminary results}

The following facts immediately follow from the Chernoff bound 
for the hypergeometric distribution \cite{luczak}.

\begin{lemma}
  \label{lem:degree-chernoff}
  For any $\varepsilon > 0$, there exists $n_0$ such that if $D$ is a 
  digraph on $n \ge n_0$ vertices, $S \subseteq V(D)$,
  $m \le |S|$ and $c := m/|S|$ then 
  there exists $T \subseteq S$ of order $m$ such that for every $v \in V$
\begin{align*}
||N^{\pm}(v) \cap T| - c|N^{\pm}(v) \cap S||&\leq \varepsilon n  ~~\text{ and}\\ 
||N^{\pm}(v) \cap \left(S \setminus T\right)| - (1-c)|N^{\pm}(v) \cap S||&\leq \varepsilon n.
  \end{align*}
\end{lemma}

We will need the following theorem and corollary and an additional lemma. 

\begin{theorem}[Moon, Moser \cite{MM}]
  \label{thm:moon_moser}
  If $G$ is a balanced bipartite graph on $n$ vertices such that
  for every $1 \le k \le n/4$
  there are less than $k$ vertices $v$ such that
  $\deg(v) \le k$ then $G$ has a Hamiltonian cycle.
\end{theorem}

\begin{corollary}
  \label{cor:mm}
  Let $G$ be a $U,V$-bipartite graph on $n$ vertices
  such that $n$ is sufficiently large and $0 \le |U| - |V| \le 1$
  and let $C \ge 3$ be a positive integer.
  If $n$ is even, let $a \in U$ and $b \in V$
  and if $n$ is odd, let $a,b \in U$.
  If $\delta(G) > 2C$ and $\deg(v) > 2n/5$ for 
  all but at most $C$ vertices $v$ 
  then $G$ has a Hamiltonian path with ends $a$ and $b$.
\end{corollary}
\begin{proof}
  If $n$ is even then
  iteratively pick $v_0 \in N(b) - a$, $v_1 \in N(v_0) - b$ and
  $v_2 \in N(a) - b - v_1$
  and set $R = \{a, b, v_0, v_1, v_2\}$.
  If $n$ is odd then 
  iteratively pick $v_1 \in N(a)$ and $v_2 \in N(b) - v_1$. 
  and set $R = \{a, b, v_1, v_2\}$.
  In both cases, we can select $v_1, v_2$ to have degree 
  greater than $2n/5$.
  Applying Theorem~\ref{thm:moon_moser} to the graph formed by
  removing $R$ from the graph and adding a new vertex to $V$ which is adjacent to $N(v_1) \cap N(v_2) \setminus R$ completes the proof.
\end{proof}

\begin{definition}
Let $S$ be a star with $k$ leaves.  If every edge of $S$ is oriented away from the center, we say $S$ is a $k$-out star, if every edge is oriented towards the center, we say $S$ is a $k$-in star.
\end{definition}

\begin{lemma}\label{2stars}
Let $G$ be a directed graph on $n$ vertices and let $1\leq d\leq D\ll n$.  If $\delta^+(G)\geq d$ and
$\Delta^-(G)\leq D$, then $G$ has at least 
$\frac{(d-1)n-4(d-1+D)}{3(d+D-1)}$ disjoint $2$-in-stars together with two independent edges.
\end{lemma}

\begin{proof}
We start by noting that since $\delta^+(G)\geq d\geq 1$ and $\Delta^-(G)\leq D$ there is a matching of size at least $2$.  Let $M$ be a maximum collection of two independent edges together with $m\geq 0$ vertex disjoint $2$-in stars and let $L=V(G)\setminus V(M)$.  Note that $\sum_{v\in L}\deg^+(v,L)\leq |L|=n-3m-4$ otherwise $\sum_{v\in L}\deg^-(v,L)=\sum_{v\in L}\deg^+(v,L)>|L|$ would give a $2$-in star disjoint from $M$.  Thus 
\begin{align*}
(d-1)(n-3m-4)\leq d(n-3m-4)-\sum_{v\in L}\deg^+(v, L)\leq \vec{e}(L,M)\leq (3m+4)D
\end{align*}
which gives $m\geq \frac{(d-1)n-4(d-1+D)}{3(d+D-1)}$.  
\end{proof}

\subsection{Finding the ADHC}

Looking ahead (in what will be the main case), we are going to distribute
vertices from $Z$ to the sets $X_1', X_2', Y_1', Y_2'$ to make sets $X_1, X_2,
Y_1, Y_2$.  Then we are going to partition each of the sets $X_1=X_1^1\cup
X_1^2$, $X_2=X_2^1\cup X_2^2$, $Y_1=Y_1^1\cup Y_1^2$, and $Y_2=Y_2^1\cup
Y_2^2$ (so that each set is approximately split in half). Then we are going to
look at the bipartite graphs induced by edges from $X_2^1\cup Y_1^1$ to $X_1^1\cup Y_1^2$ and from $X_1^2\cup Y_2^2$ to $X_2^2\cup Y_2^1$ respectively (see Figure \ref{partition}).  By the degree conditions for $X_1', X_2', Y_1', Y_2'$, these bipartite graphs will be nearly complete, however we must be sure that the vertices from $Z$ each have degree at least $\gamma n$ in the bipartite graph.  This next claim shows that the vertices of $Z$ can be distributed so that this condition is satisfied.  

\begin{figure}[ht]
\centering
\subfloat[]{
\scalebox{.7}{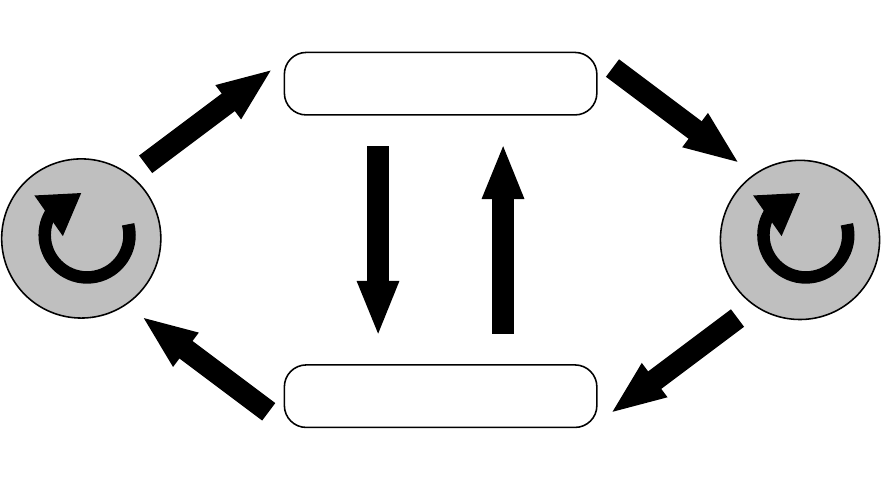}
\label{before}
}~~~~~~~~~
\subfloat[]{
\scalebox{.72}{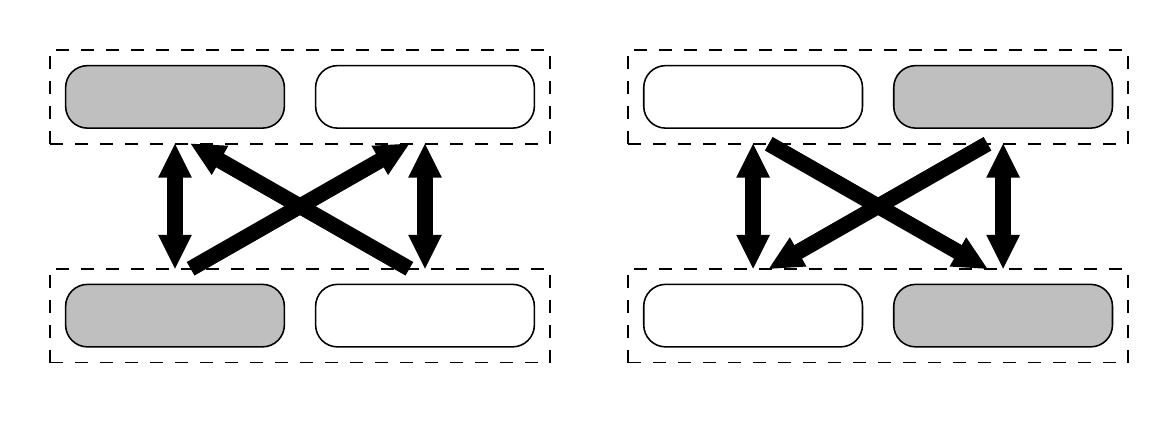}
\label{after}
}
\caption[]{The objective partition, before and after.}
\label{partition}
\end{figure}

\begin{definition}
For $z\in Z$ and $A,B\in \{X_1', X_2', Y_1', Y_2'\}$, we say $z\in Z(A,B)$ if
$\deg^+(z, B)\geq 5 \gamma n$ and $\deg^-(z, A)\geq 5 \gamma n$.
\end{definition}

\begin{claim}
  \label{clm:distrib}
Every vertex in $Z$ belongs to at least one of the following sets:
\begin{enumerate}
\item $Z(X_i', X_i')$,

\item $Z(Y_i', Y_i')$,

\item $Z(X_i', X_{3-i}')$,

\item $Z(Y_i', Y_{3-i}')$,

\item $Z_1:=\displaystyle \bigcap_{1 \le i,j \le 2} Z(Y_i', X_j')$ or
\item $Z_2:=\displaystyle \bigcap_{1 \le i,j \le 2} Z(X_i', Y_j')$.

\end{enumerate}

\end{claim}

\begin{proof}
Let $v\in Z$ and suppose that $v$ is in none of the sets $(i)-(iv)$.  Note that $v$ must have at least $(n-|Z|)/4$ out-neighbors in some set $A\in \{X_1',X_2',Y_1',Y_2'\}$.  

Assume $A=X_i'$ for some $i=1,2$.  
Because of the degree condition and the fact that
$v$ is in none of the sets $(i)-(iv)$, we have 
\begin{align*}
  \deg^-(v, Y_1\cup Y_2)&\geq n-10\gamma n-|Z|\geq (1-11\gamma)n,\text{ which implies} \\
  \deg^+(v, X_1\cup X_2)&\geq n-10\gamma n-|Z|\geq (1-11\gamma)n.
\end{align*}
This implies, $||X_1 \cup X_2|-n|,||Y_1 \cup Y_2|-n|\leq  11 \gamma n$.
With Proposition~\ref{prop:preprocess}, we have that 
$(1/2-6\gamma)n\leq |X_1|,|X_2|,|Y_1|,|Y_2| \leq (1/2 + 6\gamma)n$ so $v \in Z_1$.

If $A=Y_i'$ for some $i=1,2$,
the previous argument
(with the symbol $X$ exchanged with the symbol $Y$) gives us that $v \in Z_2$.
%Without loss of generality suppose $i=1$.  Because $v$ is in none of the sets $(i)-(iv)$, we have $\deg^-(v, X_1\cup X_2)\geq n-2\gamma n-|Z|\geq (1-3\gamma)n$.  Now because $v$ is in none of the sets $(i)-(iv)$, we have $\deg^+(v, Y_1\cup Y_2)\geq n-2\gamma n-|Z|\geq (1-3\gamma)n$.
\end{proof}

Since a vertex may be in multiple sets $(i)-(vi)$, we assign each vertex to the first set it is a member of in the ordering $(i)-(vi)$. 
%(this will have some. 
%LD This will make Case 2.1 easier.
%however we avoid picking $Z(X_1', X_2')$ unless we have no choice (for technical reasons which will not become evident until later).  
Now we distribute vertices from $Z$.
%based on how they will be used in Claim \ref{have2edges}.  

\begin{procedure}(Distributing the vertices from $Z$)
  \label{proc:distrib}
  For $1 \le i \le 2$, set
\begin{itemize}
  \item $X_i:=X_i'\cup Z(X_{3-i}', X_{3-i}')\cup Z(Y_i', Y_{3-i}')$ and
  \item $Y_i:=Y_i'\cup Z(Y_i',Y_i')\cup Z(X_{3-i}', X_i')\cup Z_i$. 
\end{itemize}
\end{procedure}

By Claim~\ref{clm:distrib}, $\{X_1, X_2, Y_1, Y_2\}$ is a partition
of $V$. (We allow empty sets in our partitions).  Note that the vertices from $Z_1\cup Z_2$ have no obvious place to be distributed, thus our choice is arbitrary.

Call a partition of a set into two parts \emph{nearly balanced} if the
sizes of the two part differ by at most $2\beta n$.
Call a partition $\bigcup_{1 \le i,j \le 2} \{X_i^j, Y_i^j\}$ of $V$ a 
\emph{splitting} of $D$ if 
$\{X_i^1, X_i^2\}$ is a nearly balanced partition of $X_i$
and $\{Y_i^1, Y_i^2\}$ is a nearly balanced partition of $Y_i$.
Define $U_i:=X_{3-i}^i\cup Y_{i}^i$ and $V_i:=X_{i}^i\cup Y_{i}^{3-i}$ (see Figure \ref{partition}).
Note that, with Proposition~\ref{prop:preprocess},
$||A|- n/2|\leq 3\beta n$ for any $A \in \{U_1, U_2, V_1, V_2\}$.
Furthermore, if $u \in U_i \setminus Z$, by Proposition~\ref{prop:preprocess},
$\deg^+(u, X_{i}' \cup Y_{i}') \ge |X_{i}' \cup Y_{i}'| - 4\alpha^{1/3}$,
so
\begin{equation}
  \label{eq:out_U_V}
  \deg^+(u, V_i) \ge |V_i| - 4\alpha^{1/3} - |Z| \ge |V_i| - 2\beta n.
\end{equation}
Similarly, if $v \in V_i \setminus Z$, then
\begin{equation}
  \label{eq:in_V_U}
  \deg^-(v, U_i) \ge |U_i| - 2 \beta n.
\end{equation}

Let $G$ be the bipartite graph on vertex sets $U_1 \cup U_2, V_1 \cup V_2$ 
such that $\{u,v\}\in E(G)$ if and only if $u\in U_1 \cup U_2, v\in V_1 \cup V_2$, 
and $(u,v)\in E(D)$. Let $G_i:=G[U_i, V_i]$  
and $Q_i = \{v \in V(G_i): \deg_G(v) < (1-\gamma)n/2\}$.
Call a splitting \emph{good} if $\delta(G_i) \ge \gamma n$ and 
$|Q_i| \le \beta n$ for $i \in 1,2$.
If $x \in X_i$ is mapped to some $X_i^j$ 
we say that $x$ is \emph{preassigned} to $X_i^j$.
Similarly, if $y \in Y_i$ is mapped to some $Y_i^j$ 
we say that $y$ is \emph{preassigned} to $Y_i^j$.

\begin{claim}\label{randomsplit}
  If $P$ is a set of preassigned vertices such that $|P| \le \beta n$
  and for all $1 \le i, j \le 2$,
  $x_i^j$ and $y_i^j$ are non-negative integers such that:
  \begin{enumerate}
    \item
      $x_i^j$ and $y_i^j$ are at least as large as the number
      of vertices preassigned to $X_i^j$ and $Y_i^j$ respectively;
    \item
      $x_i^1 + x_i^2 = |X_i|$ and $y_i^1 + y_i^2 = |Y_i|$; and 
    \item
      $||X_i|/2-x_i^j|, ||Y_i|/2-y_i^j| \leq \beta n$
  \end{enumerate}
  then there exists a good splitting of $V$
  such that $|X_i^j| = x_i^j$ and $|Y_i^j| = y_i^j$
  and every vertex in $P$ is in its preassigned set.
\end{claim}
\begin{proof}
  We can split $X_i\setminus P$ and $Y_i \setminus P$ so
  that, after adding every vertex in $P$ to its preassigned set,
  $|X_i^j| = x_i^j$ and $|Y_i^j| = y_i^j$.
  When $|X_i| \ge 5 \gamma n$,
  by Lemma~\ref{lem:degree-chernoff},
  we can also ensure that for every $v \in V$,
  \begin{align*}
    |N^{\pm}(v) \cap X_i^j| \ge |N^{\pm}(v) \cap (X_i \setminus P)|\frac{x_i^j - |P|}{|X_i \setminus P|} - \alpha n 
    &\ge \left(|N^{\pm}(v) \cap X_i| - \beta n\right)(1/2 - 2\beta n/|X_i|) - \alpha n\\ 
    &\ge |N^{\pm}(v) \cap X_i|/2 - \gamma n,
  \end{align*}
since $2\beta / 5\gamma \ll \gamma$. 
  By a similar calculation, if $|Y_i| \ge 5 \gamma n$
  we can partition $Y_i$ so that
  $|N^{\pm}(v) \cap Y_i^j| \ge |N^{\pm}(v) \cap Y_i|/2 - \gamma n$ 
  for every $v \in V$.

  Let $v \in V(G_i)$ for some $i \in \{1,2\}$.
  If $v \in Z$, by the previous calculation, Claim~\ref{clm:distrib} 
  and Procedure~\ref{proc:distrib}, $d_{G_i}(v) \ge \gamma n$.
  If $v \notin Z$, by \ref{eq:out_U_V} and \ref{eq:in_V_U}, 
  $d_{G_i}(v) \ge (1 - \gamma)n/2$.
  Therefore, $\delta(G_i) \ge \gamma n$ and $|Q_i| \le \beta n$.
\end{proof}

\begin{proposition}\label{reduce}
  If there exists a good splitting of $D$ and two independent edges
  $uv$ and $u'v'$ such that either
\begin{enumerate}
  \item $u \in U_1$, $v\in V_2$, 
    $u'\in U_2$, $v'\in V_1$ and $|U_i|=|V_{3-i}|$ for $i=1,2$; or
  \item there exists $i=1,2$ such that 
    $u,u'\in U_i$, $v,v'\in V_{3-i}$, 
    $|U_i|=|V_i|+1$ and $|V_{3-i}|=|U_{3-i}|+1$
\end{enumerate}
then $D$ contains an ADHC.  
\end{proposition}
\begin{proof}
Apply Corollary~\ref{cor:mm} to get a Hamiltonian path $P_i$ in $G_i$ so
that the ends of $P_1$ and $P_2$ are the vertices $\{u,u', v, v'\}$.
These paths and the edges $uv$ and $u'v'$ correspond to an ADHC in $D$.
\end{proof}

Note that the edges $uv$ and $u'v'$ played a special role in the previous proposition.  Now we discuss what properties these edges must have and how we can find them (this will be the bottleneck of the proof in each case and is the only place where the exact degree condition will be needed).

\begin{definition}
  Let $uv$ be an edge in $D$.  We call $uv$ a \emph{connecting edge} if
  for some $i = 1,2$,
%\begin{enumerate}
%\item 
  $u\in X_i$ and either $v\in X_i$ or $v\in Y_i$; or
%\item $u\in X_2$ and either $v\in X_2$ or $v\in Y_2$
%\item 
$u\in Y_i$ and either $v\in Y_{3-i}$ or $v\in X_{3-i}$.
%\item $u\in Y_2$ and either $v\in Y_1$ or $v\in X_1$
%\end{enumerate}

\end{definition}

Basically, connecting edges are edges which do not behave like edges in the graph shown in Figure \ref{before}.

The following simple inequalities are used to help find connecting
edges and follow directly from the degree condition.
For any $A \subseteq V$ and $v \in A$
\begin{align}
  \label{deginside}
  \deg^0(v, A) &\geq n-|\overline{A}| \\
  \label{degoutside}
  \deg^0(v, \overline{A}) &\geq n-(|A|-1)=n + 1-|A|.
\end{align}

At this point, we split the proof into two main cases depending on the order
of the sets $Y_1$ and $Y_2$. 
%$Y_1$ and $Y_1$ which of the following cases we are in.

\noindent
\textbf{Case 1:} $\min\{|Y_1|, |Y_2|\} > \beta n$

Without loss of generality, suppose $|X_1 \cup Y_1| \ge |X_2 \cup Y_2|$.

If $|X_1 \cup Y_1| > n$ and $|X_1|\leq 2$, then let $X_1''\subseteq \{v\in X_1: \deg^-(v, Y_2\cup X_1'')\geq 5\gamma n\}$ be as large as possible subject to $|X_1''|\leq |X_1\cup Y_1|-n$.  Reset $X_1:=X_1\setminus X_1''$ and $Y_2:=Y_2\cup X_1''$.  If $|X_1\cup Y_1|=n$ and $|X_1|=1$, say $X_1=\{v_1\}$, then if $\deg^-(v_1, Y_{2})\geq 5\gamma n$ and there exists $v_{2}\in X_{2}$ such that $\deg^-(v_{2}, Y_1)\geq 5\gamma n$, then we reset $X_i:=X_i\setminus\{v_i\}$ and $Y_i:=Y_i\cup \{v_{3-i}\}$ for $i=1,2$.

It is easy to check that 
the conclusions of Claim~\ref{randomsplit} still hold with the possibly
redefined sets $\{X_1, X_2, Y_1, Y_2\}$.  Furthermore, after these modifications, we still
have that $|X_1 \cup Y_1| \ge |X_2 \cup Y_2|$ and the following
two conditions are satisfied:
\begin{align}
\label{eq:mod_case_1_2}
  &\text{If $|X_1| = 1$, then there exists $i\in [2]$ such that for all $v\in X_i$, $\deg^-(v, Y_{3-i}) < 5 \gamma n$.}\\
  \label{eq:mod_case_1_1}
  &\text{If $|X_1 \cup Y_1| > n$ and $|X_1| \le 2$, then 
  for every $v \in X_1$, $\deg^-(v, Y_2) < 5 \gamma n$.} 
\end{align}

\begin{claim}\label{Case2edges}
For each $i=1,2$, there exists a partition of 
$X_i$ as $\{X_i^1,X_i^2\}$ with $||X_i^1| - |X_i^2|| \le \alpha n$
and $W_i:=Y_i\cup X_1^i\cup X_2^i$ such that
either
\begin{enumerate}
\item $|W_1|$, $|W_2|$ are odd and there are two independent connecting edges
  directed from $W_j$ to $W_{3-j}$ for some $j=1,2$; or
\item $|W_1|,|W_2|$ are even and there are two independent connecting edges, one directed from $W_1$ to $W_2$ and the other directed from $W_2$ to $W_1$.
\end{enumerate}
\end{claim}

\begin{proof}
%Without loss of generality suppose $|X_1\cup Y_1|\leq |X_2\cup Y_2|$.

\noindent
\textbf{Case 1} ($|X_1\cup Y_1|=|X_2\cup Y_2|$).  For all $u\in Y_1$ and $u'\in Y_2$, we have $\deg^0(u, X_2\cup Y_2), \deg^0(u', X_1\cup Y_1)\geq 1$ by \eqref{degoutside}.  From this we get independent edges $uv$ and $u'v'$ with $v\in X_2\cup Y_2$ and $v'\in X_1\cup Y_1$.  We would be done unless $n$ is odd and $X_1\subseteq \{v'\}$ and $X_2\subseteq \{v\}$, as otherwise we could obtain the partition $W_i:=Y_i\cup X_1^i\cup X_2^i$ for $i=1,2$ with $u,v'\in W_1$ and $v,u'\in W_2$ and $|W_1|, |W_2|$ even. If there exists $u''\in Y_1$ having an outneighbor in $X_2\cup Y_2$ different from $v$, then we would be done, likewise if there exists $u''\in Y_2$ having an out-neighbor in $X_1\cup Y_1$ different than $v'$. 
If say $X_1=\emptyset$, then choosing $v''\in Y_2$ with $v''\neq v$ and using the fact that $\deg^-(v'', Y_2)\geq 1$ we can choose $u''\in Y_1$ different from $u$ (reselecting $u$ if necessary) in which case we would be done.  
So we must have that $X_1=\{v'\}$ and $X_2=\{v\}$ with $\deg^-(v, Y_1)\geq |Y_1|$ and $\deg^-(v', Y_2)\geq |Y_2|$, but this contradicts \eqref{eq:mod_case_1_2}.
%We have $\deg^0(v', Y_2)\geq \gamma n$ and $\deg^0(v, Y_1)\geq \gamma n$, so we can instead consider the partition $W_1':=Y_1\cup X_2$ and $W_2':=Y_2\cup X_1$.  Now we simply use the fact that $\delta^+(W_1', W_2')\geq 1$ and $\delta^-(W_2', W_1')\geq 1$ to get independent edges $uv$ and $u'v'$ with $u,u'\in W_1'$ and $v,v'\in W_2'$.  Since every vertex in $W_1'$ has big out-degree in $W_1'$ and every vertex in $W_2'$ has big in-degree in $W_2'$, these edges satisfy the property of connecting edges.  

\noindent
\textbf{Case 2} ($|X_1\cup Y_1|>|X_2\cup Y_2|$).  
%Let $X_2'=\{x\in X_2:\deg^-(x, Y_1)\geq \gamma n\}$.  If $|X_2\cup Y_2|-|X_2'|\leq n$, then we can move a subset $X_2''$ of vertices from $X_2'$ to $Y_1$ so that $|X_1\cup Y_1\cup X_2''|=n=|(X_2\cup Y_2)\setminus X_2''|$.  This puts us back in Case 1; so suppose $|X_2\cup Y_2|-|X_2'|> n$ and reset $Y_1:=Y_1\cup X_2'$ and $X_2:=X_2\setminus X_2'$.  We still have $|X_1\cup Y_1|<|X_2\cup Y_2|$, but now we have $\Delta^-(X_2, Y_1)<\gamma n$.
By the case, we can choose distinct $u,u'\in Y_2$ such that $\deg^0(u, X_1\cup Y_1),\deg^0(u', X_1\cup Y_1)\geq 2$ by \eqref{degoutside}. Thus we can choose distinct 
$v\in N^+(u)\cap (X_1\cup Y_1)$ and $v'\in N^+(u')\cap (X_1\cup Y_1)$, with a preference for choosing vertices in $Y_1$.
For $i = 1,2$, let $\{X_i^1, X_i^2\}$ be a partition of 
$X_i$ such that $||X_i^1| - |X_i^2|| \le \alpha n$
and $W_i:=Y_i\cup X_1^i\cup X_2^i$ with $u,u'\in W_1$ and $v,v'\in W_2$.
If this can be done so that $|W_1|$ and $|W_2|$ are odd then we are done, 
so suppose not.
Then it must be the case that $X_2=\emptyset$ and $X_1 \subseteq \{v,v'\}$.  If $X_1\neq \emptyset$, then every vertex in $Y_2$ has an out-neighbor in $X_1$ which implies that $\deg^-(v, Y_2)\geq |Y_2|/2$ for some $v\in X_1$, contradicting \eqref{eq:mod_case_1_1}.
%$\Delta^-(X_2, Y_1)<\gamma n$.  
So suppose $X_1=\emptyset$.  Now we can finish by choosing $v''\in Y_2$ distinct from $u$ and letting $u''\in (N^-(v'')\cap Y_1)\setminus \{v\}$.  
%Hence, $W_1=Y_1$ and $W_2=X_2\cup Y_2$.  
%%\begin{equation}\label{degreetoW2}
%%\deg^0(w, W_i)\geq n+1-(|W_{3-i}|-1)=2+(n-|W_{3-i}|) \text{ for all } w\in
%%W_{3-i}.  
%%\end{equation}
%Suppose $|W_1|\geq |W_2|$, by \eqref{degoutside} we have $\deg^0(w, W_1)\geq 1$ for all $w\in W_2$.  In this case we choose $u''\in Y_2\setminus \{v\}$ and then $v''\in N^+(u'')\cap (W_1\setminus\{u\})$ giving us the desired connecting edges $uv$ and $u''v''$.  
%% TNM - not needed since we know |W_1| \ge |W_2|
%So suppose $|W_1|<|W_2|$.  If $|W_2|>n+1$, then  
%$\deg^0(w, W_2)\geq 3$ for all $w\in W_1$ by \eqref{degoutside}.  
%Letting $v''\in W_1\setminus \{u\}$ and 
%$u''\in N^-(v'')\cap (W_2\setminus (X_2\cup \{v\}))$ (recall
%$X_2 \subseteq \{v, v'\}$ ) gives us the desired connecting edges $uv$ and $u''v''$.  So suppose $|W_2|=n+1$
\end{proof}

By Claim \ref{Case2edges} and Proposition \ref{prop:preprocess} for $i=1,2$ we have 
$||X_{1}^i| - |X_2^i|| \le \alpha n + 3 \alpha^{2/3} n$.
So since $|Y_i| \ge \beta n$
we can assume that after we apply Claim~\ref{randomsplit},
$||U_{i}| - |V_i|| \le 1$.

Let $uv$ and $u'v'$ be the connecting edges from Claim~\ref{Case2edges}.
Suppose Claim~\ref{Case2edges}.(i) holds and
fix $i \in \{1,2\}$ so that
$u,u' \in W_i$ and $v,v' \in W_{3-i}$.
Preassign $u, u', v$ and $v'$ so that, 
after splitting $D$ with Proposition~\ref{randomsplit},
$u, u' \in U_i$ and $v, v' \in V_{3 - i}$.
Since $|W_1|$ and $|W_2|$ are odd,
we can ensure that $|U_i| = |V_i| + 1$ and
$|V_{3-i}| = |U_{3-i}| + 1$.
We can then apply Proposition~\ref{reduce}.(ii) to find an ADHC.
Now suppose Claim~\ref{Case2edges}.(ii) holds
and let $u,v' \in W_1$, $v,u' \in W_2$
so that $uv$ and $u'v'$ are the connecting edges.
Preassign $u, u', v$ and $v'$ so that, 
after splitting $D$ with Proposition~\ref{randomsplit},
$u \in U_1$,$v \in V_2$, $u' \in U_2$ and $v' \in V_1$.
Since $|W_1|$ and $|W_2|$ are even,
we can apply Proposition~\ref{reduce}.(i) to find an ADHC.

%
%Note about Theorem \ref{2factor}: 
%%If we only had $\delta^0(D)\geq n$, then Claim \ref{Case2edges} may not hold.  If it doesn't, then 
%We can either find a partition $\{W_1,W_2\}$ so that $|W_1|$ and $|W_2|$ are
%even -- in which case we find a spanning ADC in both subgraphs.  If $|W_1|$ and $|W_2|$ are odd and without loss of generality $|W_1|\leq |W_2|$, 
%we try to find two independent connecting edges from $W_1$ to $W_{2}$ and then construct an ADHC.
%Otherwise, by the degree condition, there exist
%two vertices in $W_1$ with a common out-neighbor $z\in W_{2}$.
%We can use this extra vertex $z$ to find a spanning ADC of $W_1\cup \{z\}$ and
%a spanning ADC of $W_{2}\setminus\{z\}$.

\noindent
\textbf{Case 2: $\min\{|Y_1|, |Y_2|\} \le \beta n$}

%By Proposition \ref{prop:preprocess}, we see that $|X_1|,|X_2|\geq n-3\beta n$.  
Without loss of generality, suppose $|X_1|\geq |X_2|$.  If $|X_1\cup Y_1|>n$, then let $$X_1''\subseteq \{v\in X_1:\deg^-(v, X_1)\geq 5 \gamma n\}\cup \{v\in Y_1: \deg^-(v, X_1)\geq 5\gamma n\}$$ 
be as large as possible subject to $|X_1''| \le |X_1\cup Y_1|-n$. 
Reset $X_1 :=X_1 \setminus X_1''$ and $Y_1:= Y_1\setminus X_1''$ and $X_2:=X_2\cup X_1''$.  %If $v\in X_1''$ also satisfies $\deg^+(v, X_1)\geq 5\gamma n$, then reassign $v$ to $X_2$, if $\deg^+(v, X_1)<5\gamma n$ then reassign $v$ to $Y_2$.  
If we still have $|X_1\cup Y_1|>n$, then  
%\begin{equation}
%\Delta^-(X_1)<5\gamma n
%\end{equation}
%and 
because of how we distributed the vertices in Claim \ref{clm:distrib} and Procedure \ref{proc:distrib} together with how we reassigned the vertices of $X_1''$, we have 
\begin{equation}\label{X1Y1toX1}
\Delta^-(X_1\cup Y_1, X_1)<5\gamma n.
\end{equation}

By Proposition~\ref{prop:preprocess},
$|X_1'| \le n + 2\alpha^{2/3}$ and $|Z| \le 3 \alpha^{2/3}$,
thus $|X_1''| \le 5\alpha^{3/2} \ll \beta n$.
Therefore, 
the conclusions of Claim~\ref{randomsplit} still hold with the redefined
sets $\{X_1, X_2, Y_1, Y_2\}$.

\noindent
\textbf{Case 2.1:} $|X_1| \le n$.  If $Y_1=\emptyset$ or $Y_2=\emptyset$, say $Y_1=\emptyset$, then we can split $Y_2=Y_2^1\cup Y_2^2$ so that $|X_1\cup Y_2^1|=n=|X_2\cup Y_2^2|$.  In this case we can directly find the ADHC by only considering edges from $X_1\cup Y_2^1$ to $X_2\cup Y_2^2$. So suppose $Y_1\neq \emptyset$ and $Y_2\neq \emptyset$.

%Now suppose $Y_1 \neq \emptyset$ and $Y_2 \neq \emptyset$.
Suppose $|X_1\cup Y_1|=|X_2\cup Y_2| = n$.  We first note that two independent connecting edges $uv, u'v'$ will allow us to either preassign $u,u'$ so that $u\in U_1$ and $u'\in U_2$ and $v,v'$ so that $v \in V_1$ and $v' \in V_2$, or preassign $u,u'$ so that $u, u'\in U_i$ and $v,v'$ so that $v,v'\in V_{3-i}$ (this is possible since $Y_1,Y_2\neq \emptyset$).
In the first case 
%Because $|X_1 \cup \{u,v\}|, |X_2 \cup \{u',v'\}| \le n$,
we can apply Claim~\ref{randomsplit},
so that $|U_1| + |V_2| = |U_2| + |V_1| = n$,
$|U_1|=|V_1|$ and $|U_2|=|V_2|$; in the second case we can apply Claim \ref{randomsplit} so that $|U_{i}| + |V_{3-i}| = n+1$,
$|U_{3-i}| = |V_{3-i}| + 1$ and $|V_{i}| = |U_{i}| + 1$.
Applying Proposition~\ref{reduce}.(i) or (ii) then gives the desired ADHC.

So in this case we show that $D$ must contain two independent connecting edges (here is the only place where we make use of the fact that $D$ is not isomorphic to $F_{2n}^1$ or $F_{2n}^2$). Note that:
\begin{align}
\delta^+(Y_i, X_{3-i}\cup Y_{3-i})&\geq n-(|X_i\cup Y_i|-1)=1 ~\text{ for } i=1,2 \label{Youtdeg}
%\delta^-(Y_1, X_1\cup Y_2)+\delta^-(Y_2, X_2\cup Y_1)&\geq 2n-(|X_2\cup Y_1|-1)-(|X_1\cup Y_2|-1)=2\label{Ydegsum}
\end{align}
If there is an edge in $D[X_1]$ or an edge in $D[X_2]$; or
$|Y_1| \ge 2$ and $|Y_2| \ge 2$, then we easily obtain two independent connecting edges using \eqref{Youtdeg}.
If say $|Y_1|=1$ and $|Y_2|\geq 2$, then 
$|Y_1 \cup X_2| \le n-1$
so $\delta^-(Y_1, X_1 \cup Y_2) \ge 2$
and $\delta^-(X_1, Y_2)\geq n-|Y_1\cup X_2|\geq 1$, which
together give two independent edges.
%If say $|Y_1|,|Y_2|\geq 2$, then by \eqref{Youtdeg} we can find two independent connecting edges.  If say $|Y_1|=1$ and $|Y_2|\geq 2$, then $\delta^-(X_1, X_1\cup Y_2)\geq n-|Y_1\cup X_2|\geq 1$.  Together with \eqref{Ydegsum}, this gives two independent connecting edges.  
Finally, if $|Y_1|=1=|Y_2|$, let $\{y_i\} = Y_i$ for $i = 1,2$.  If there exists $x_1\in X_1$ and $x_2\in X_2$ such that $x_ix_{3-i}$ is not an edge for some $i\in [2]$, then because of the semi-degree condition and the fact that $X_1$ and $X_2$ are independent sets, it must be that $x_iy_i$ and $x_{3-i}y_{3-i}$ are edges, giving us two independent connecting edges. 
If there exists $x_i \in X_i$ such that $y_ix_i$ is not an edge,
then, by the semi-degree condition and the fact that $X_i$ is
an independent set, $y_{3-i}x_i$ is an edge.
Also by the semi-degree condition, $y_i$ must have an out-neighbor
in $X_{3-i}$ and, with the edge $y_{3-i}x_i$, this
gives us two independent connecting edges.
If there exists $x_i \in X_i$ such that $x_iy_{3-i}$ is not an edge, then an analogous argument gives two independent connecting edges.  
So we have proved that $D$ contains a subgraph isomorphic to $F_{2n,1}$.
Since $|X_1 \cup Y_2| = |X_2 \cup Y_1| = n$, 
the semi-degree condition implies that 
every vertex in $y \in Y_1 \cup Y_2$ is incident to at least 
two connecting edges: one oriented away from $y$ and the other oriented towards $y$.
If $\{y_1, y_2\}$ is an independent set, then we clearly
have two connecting edges, so assume that $y_iy_{3-i}$ is an edge.
If $y_{3-i}y_i$ is an edge, then since $D$ is not isomorphic to 
$F_{2n}^1$, there must exist at least one more edge in $D$.  Since $F_{2n}^1$ is an edge-maximal graph without an ADHC, $D$ must contain an ADHC.
So we can assume $y_{3-i}y_i$ is not an edge, and thus 
$y_{i}$ must have an in-neighbor $x_i$ in $X_i$ and 
$y_{3-i}$ must have an out-neighbor $x_i'$ in $X_i$.  If $x_i\neq x_i'$, then we have two independent connecting edges.  If $x_i=x_i'$, then $D$ contains a subgraph isomorphic to 
$F_{2n}^2$, and as before since $D$ is not isomorphic to $F_{2n}^2$, $D$ must contain an ADHC.

%Finally suppose $|Y_1|=1=|Y_2|$.  If there is an edge inside $X_1$ or $X_2$, then by \eqref{Youtdeg} and \eqref{Ydegsum}, we can find an independent connecting edge incident with $Y_1\cup Y_2$.  So suppose there are no edges inside $X_1$ or $X_2$.  This implies that $F_{2n,1}$ is a subgraph of $D$.  Now the only way for the vertices in $Y_1\cup Y_2$ to have semi-degree at least $n$ and satisfy \eqref{Youtdeg} and \eqref{Ydegsum} without having two independent connecting edges is for $D$ to be isomorphic to $F_{2n}^1$ or $F_{2n}^2$, which we are assuming is not the case.
 
%By \eqref{deginside}, we have $\deg^0(u, X_{1}\cup Y_{1}), \deg^0(u', X_{2}\cup
%Y_{2})\geq 1$, so choose $v\in N^+(u)\cap (X_1\cup Y_1)$ and $v'\in
%N^+(u')\cap (X_2\cup Y_2)$.
%%Since $Y_1,Y_2\neq \emptyset$, $|X_1|,|X_2|\leq n-1$, thus $|X_1\cup \{v\}|,|X_2\cup \{v'\}|\leq n$.  
%%Now we can partition $Y_1,Y_2,X_1,X_2$ so that the conditions of 
%%Proposition \ref{reduce}.(i) are satisfied.
%Preassign $u$ to $X_1^2$, $u'$ to $X_2^1$
%and $v,v'$ so that $v \in V_1$ and $v' \in V_2$. 
%Because $|X_1 \cup \{u,v\}|, |X_2 \cup \{u',v'\}| \le n$,
%we can apply Claim~\ref{randomsplit},
%so that $|U_1| + |V_2| = |U_2| + |V_1| = n$,
%$|U_1|=|V_1|$ and $|U_2|=|V_2|$.
%Applying Proposition~\ref{reduce}.(i) then gives the desired ADHC.

Now suppose $|X_i\cup Y_i|>|X_{3-i}\cup Y_{3-i}|$ for some $i=1,2$.
By \eqref{deginside}, $\deg^0(u, X_{i}\cup Y_{i})\geq 1$ for all $u\in X_i\cup
Y_i$ and $\deg^+(u, X_i\cup Y_i)\geq n-(|X_{3-i}\cup Y_{3-i}|-1)\geq 2$ for all $u\in Y_{3-i}$.  Let $u\in Y_{3-i}$, let $v_1,v_2\in N^+(u)\cap (X_i\cup Y_i)$, and let $u'\in X_i\setminus \{v_1,v_2\}$. Choose distinct $v\in N^+(u)\cap (X_i\cup Y_i)$ and $v'\in N^+(u')\cap (X_i\cup Y_i)$ with a preference for choosing
$v$ and $v'$ in $X_i$ (this can be done since we chose $u'$ distinct from $v_1,v_2$).  
If $|X_i| \le n-2$, then $|X_i \cup  \{u,u',v,v'\}| \le n+1$, and
when $i=2$, $|X_2 \cup Y_2| > |X_1 \cup Y_1|$, and $|Y_1| \ge 1$ imply that $n - 2 \ge |X_1| \ge |X_i|$.  So suppose $i=1$ and $n-1\leq |X_1|\leq n$.  
Note that in this case, for every $u \in X_1$, 
$\deg^+(u, X_1 \cup Y_1) \ge \max\{1, |Y_1| - 1\} \ge |Y_1|/2$, so
the bound in implies that there are two disjoint edges in $G[X_1]$.

So we can assume, in all cases, that $|X_i \cup \{u,u',v,v'\}| \le n+1$.
Therefore, after preassigning $u,u'$ to $X_i^{3-i}$ and 
$v,v'$ to $X_i^i$ or $Y_i^{3-i}$ as appropriate,
we can apply Claim~\ref{randomsplit} to get
$|U_{3-i}| + |V_{i}| = n+1$,
$|U_{3-i}| = |V_{3-i}| + 1$ and $|V_{i}| = |U_{i}| + 1$.
Applying Proposition~\ref{reduce}.(ii) then completes this case.

\bigskip
\noindent
\textbf{Case 2.2:} $|X_1|\geq n+1$.

Set $d=|X_1|-n$ and recall that $d \ll \beta n$. 
By \eqref{deginside}, $\delta^+(D[X_1]) \ge d $.
By the case, $X_1'' \cap X_1 = \emptyset$,
so $\Delta^-(D[X_1])<5\gamma n$ and
$\frac{(d-1)n-4(d-1+5\gamma n)}{3(d-1+5\gamma n)}\geq d-1$.
Applying Lemma \ref{2stars} gives two independent edges $uv$, $u'v'$ and a collection 
of $d-1$ vertex disjoint $2$-in stars $\{S_1, \dotsc, S_{d-1}\}$ in $D[X_1]$. 
Preassign the vertices in $S_1,\dotsc,S_{d-1}$ and
the vertices $u$ and $u'$ to $X_1^2$. 
Also, preassign $v$ and $v'$ to $X_1^1$.
Recall that $X_1^1 \cup X_1^2 \subseteq U_2 \cup V_1$,
so we can use Claim~\ref{randomsplit}, 
to get a good splitting of $D$ such that 
$|U_2| = \ceiling{n/2} + d$, $|V_1| = \floor{n/2}$, 
$|V_2| = \ceiling{n/2} - d + 1$ and $|U_1| = \floor{n/2} - 1$.
We then use Corollary~\ref{cor:mm}, to find a Hamiltonian path $P_1$ in $G_1$ 
with ends $v$ and $v'$.

We now move the roots of the stars $S_1, \dotsc, S_{d-1}$
from $U_2$ to $V_2$ and then use Corollary~\ref{cor:mm} to complete the
proof. More explicitly, we greedily find a matching $M$ between the leaves of the stars $S_1,\dotsc,S_{d-1}$
and the vertices in $V_2$ of degree at least $(1 - \gamma)n/2$ in $G_2$.
For each $1 \le i \le d -1$, 
let $a_i$ and $b_i$ be the vertices matched to 
the leaves of $S_i$ and
replace $V(S_i) \cup \{a_i, b_i\}$ in $G_2$
with a new vertex adjacent to $N_{G_2}(a_i) \cap N_{G_2}(b_i)$
minus the vertices of the stars.
Apply Corollary~\ref{cor:mm} to get a Hamiltonian path $P_2$ in
the resulting graph with ends $u$ and $u'$.
The stars $S_1, \dotsc, S_{d-1}$; the edges in $M$;
the paths $P_1$ and $P_2$; 
and the edges $uv$ and $u'v'$ correspond to an ADHC in $D$.

%
%
%Note about Theorem \ref{2factor}:  If $|X_1|=n+1$, then we find two $1$-stars (i.e. connecting edges) instead of $2$-stars.  If $|X_1|\geq n+2$, then we do the same thing as above.  Furthermore, provided $|X_1|\leq n$ we don't need connecting edges in Case 2 because here we will find two ADC's which span $D$ instead of an ADHC.

\section{Conclusion}\label{conclusion}

We end with the following conjecture which along with Theorem \ref{main} would provide a full generalization of Dirac's theorem to directed graphs with respect to minimum semi-degree.

\begin{conjecture}
Let $D$ be a directed graph on $n$ vertices and let $\vec{C}$ be an orientation of a cycle on $n$ vertices which is not anti-directed.  If $\delta^0(D)\geq \frac{n}{2}$, then $\vec{C}\subseteq D$.
\end{conjecture}

We believe that the methods developed in this paper along with the ideas in \cite{HT} and \cite{HT2} provide an approach to this problem.  We intend to carry out this program in a subsequent paper.

\end{document}